\documentclass[11pt]{article}

\pdfoutput=1

\title{Connecting $3$-manifold triangulations with unimodal sequences of elementary moves}
\date{\today}

\author{
 Benjamin A. Burton\\
 The University of Queensland\\
 \email{bab}{maths.uq.edu.au}
 \and
 Alexander He\\
 Oklahoma State University\\
 \email{alex.he}{okstate.edu}
}

\usepackage{amsmath, amssymb, amsthm, mathtools, thmtools, thm-restate, caption, subcaption, anyfontsize, microtype}
\usepackage{nth}

\usepackage[bottom]{footmisc}

\usepackage[a4paper, margin=2cm]{geometry}

\usepackage{graphicx}
\graphicspath{
	{./MonotonicFigures/}
	{./MonotonicV2Figures/}
 }
\usepackage{tikz}
\usetikzlibrary{calc} % For coordinate calculations
\usetikzlibrary{arrows.meta} % For more arrow options

\usepackage[breaklinks,backref=page]{hyperref}

\usepackage[shortlabels]{enumitem}
\setlist[itemize]{leftmargin=*, noitemsep}
\setlist[enumerate]{leftmargin=*, noitemsep}
\setlist[description]{leftmargin=*, labelwidth=*, noitemsep}

\theoremstyle{plain}
\newtheorem{theorem}{Theorem}
\newtheorem{lemma}[theorem]{Lemma}

\newtheorem{observation}[theorem]{Observation}

\newtheorem{conjecture}[theorem]{Conjecture}

\theoremstyle{definition}
\newtheorem{remark}[theorem]{Remark}
\newtheorem{notation}[theorem]{Notation}

\declaretheoremstyle[
	headfont=\normalfont\bfseries,
	numbered=yes,
	bodyfont=\normalfont,
	qed={$\blacksquare$},
	spaceabove=1em,
	spacebelow=1em,
]{qeddef}

\declaretheorem[
	style=qeddef,
	title=Definition,
	sibling=theorem,
]{definition}

\makeatletter
\renewenvironment{proof}[1][\proofname] {\par\pushQED{\qed}\normalfont\topsep6\p@\@plus6\p@\relax\trivlist\item[\hskip\labelsep\itshape\bfseries#1\@addpunct{.}]\ignorespaces}{\popQED\endtrivlist\@endpefalse}
\makeatother

\newcommand{\email}[2]{\href{mailto:#1@#2}{\textsf{#1\hspace{1pt}$@$\hspace{1pt}#2}}}

\begin{document}

\maketitle

\begin{abstract}
A key result in computational $3$-manifold topology is that any two triangulations of the same $3$-manifold are
connected by a finite sequence of bistellar flips, also known as Pachner moves.
One limitation of this result is that little is known about the structure of this sequence---knowing
more about the structure could help both proofs and algorithms.
Motivated by this, we consider sequences of moves that
are ``unimodal'' in the sense that they break up into two parts:
first, a sequence that monotonically increases the size of the triangulation;
and second, a sequence that monotonically decreases the size.
We prove that any two one-vertex triangulations of the same $3$-manifold, each with at least two tetrahedra,
are connected by a unimodal sequence of 2-3 and 2-0 moves.
We also study the practical utility of unimodal sequences;
specifically, we implement an algorithm to find such sequences, and use this algorithm to perform some detailed computational experiments.
\end{abstract}
\paragraph{Keywords}Computational topology, $3$-manifolds, triangulations, special spines,
elementary moves, bistellar flips, Pachner moves

\paragraph{\textbf{\textup{Acknowledgements}}}
AH was supported by an Australian Government Research Training Program Scholarship.
We thank the anonymous reviewers for their helpful feedback, especially for suggesting the name ``unimodal'', which is much better than our previous terminology.

\section{Introduction}\label{sec:intro}

Given two triangulated $3$-manifolds, the \emph{homeomorphism problem} asks:
are these two $3$-manifolds homeomorphic?
This is one of the oldest and most important problems in computational low-dimensional topology.
Although we now know that there is an algorithm for this problem~\cite{Kuperberg2019, ScottShort2014}, the algorithm is far too complicated
to be implemented, and the best known upper bound on the running time is a tower of exponentials~\cite{Kuperberg2019}.
Nevertheless, we have a variety of techniques that, in practice, often allow us to solve the homeomorphism problem with real data.
\begin{itemize}
\item If the answer is ``no'' (that is, the $3$-manifolds are \emph{not} homeomorphic), then we can often distinguish
the given $3$-manifolds using efficiently computable invariants, such as homology groups
or the Turaev-Viro invariants~\cite{BMS2018,TuraevViro1992}.
\item If the answer is ``yes'' (that is, the $3$-manifolds \emph{are} homeomorphic), then we have a range of tools
(some of which we mention shortly) that we can use to verify that the given $3$-manifolds are homeomorphic.
\end{itemize}
One important application of these techniques is in building censuses of $3$-manifolds:
when we generate all possible triangulations up to a certain number of tetrahedra,
many $3$-manifolds may have more than one representative among these triangulations,
so we need to be able to identify (and hence eliminate) such duplicates.

As we discuss in Section~\ref{subsec:introMoves}, elementary moves---which involve modifying a small piece of a triangulation without
changing the underlying $3$-manifold---have a significant role to play in both the ``yes'' and ``no'' directions of the homeomorphism problem.
More broadly, elementary moves are often crucial for practical computation with $3$-manifolds,
simply because they provide an easy and often surprisingly effective way to ``improve'' triangulations;
for instance, since the complexity of many $3$-manifold algorithms scales exponentially (or worse) with the number of tetrahedra
in the input triangulation, using elementary moves to reduce the number of tetrahedra is often
critical for obtaining computational results within reasonable time and memory constraints.

Motivated by all these applications, our goal in this paper is, roughly speaking, to gain a finer theoretical understanding of elementary moves.
Before we discuss this in more detail, we first review the elementary moves that we will use.

\subsection{Elementary moves on triangulations}\label{subsec:introMoves}

The most well-known elementary moves are the \emph{Pachner moves} (also often called \emph{bistellar flips}).
For $3$-dimensional triangulations, the four Pachner moves---illustrated in Figures~\ref{fig:pachnerMove} and~\ref{fig:1-4Move}---are:
\begin{itemize}
\item the \textbf{2-3 move}, which replaces two distinct tetrahedra attached along a triangular face
with three distinct tetrahedra attached around an edge;
\item the \textbf{3-2 move}, which is the inverse of the 2-3 move;
\item the \textbf{1-4 move}, which subdivides a single tetrahedron into
four distinct tetrahedra attached around a vertex; and
\item the \textbf{4-1 move}, which is the inverse of the 1-4 move.
\end{itemize}
Pachner showed that any two triangulations of the same $3$-manifold are
connected by a finite sequence of these moves~\cite{Pachner1991}.

\begin{figure}[htbp]
\centering
	\begin{tikzpicture}
	% Subpictures
	\node[inner sep=0pt] (Before) at (0,0)
		{\includegraphics[scale=0.7]{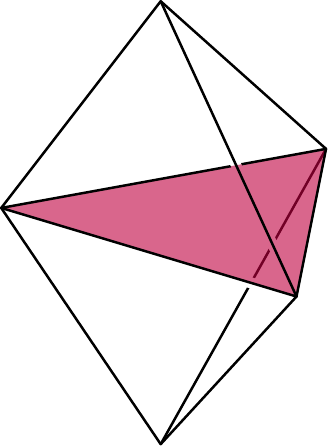}};
	\node[inner sep=0pt] (After) at (5.5,0)
		{\includegraphics[scale=0.7]{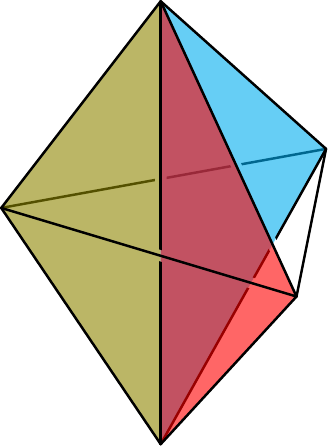}};

	% Arrows with labels
	\begin{scope}[very thick, line cap=round, -{Stealth}]
	\draw ($(Before.east)+(0,0.3)$)
		-- ($(After.west)+(0,0.3)$)
		node[midway, above, inner sep=3pt] {2-3};
	\draw ($(After.west)+(0,-0.3)$)
		-- ($(Before.east)+(0,-0.3)$)
		node[midway, below, inner sep=3pt] {3-2};
	\end{scope}
	\end{tikzpicture}
\caption{The 2-3 and 3-2 moves.}
\label{fig:pachnerMove}
\end{figure}

\begin{figure}[htbp]
\centering
	\begin{tikzpicture}
	% Subpictures
	\node[inner sep=0pt] (Before) at (0,0)
		{\includegraphics[scale=0.45]{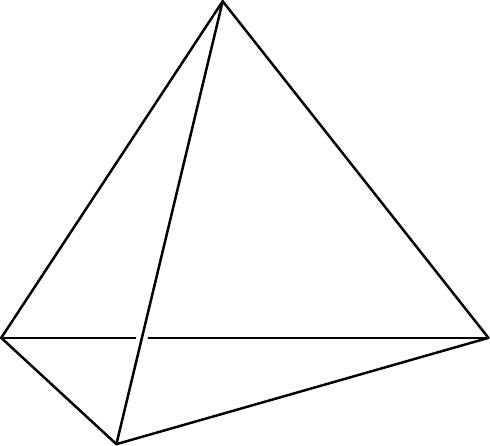}};
	\node[inner sep=0pt] (After) at (5.5,0)
		{\includegraphics[scale=0.45]{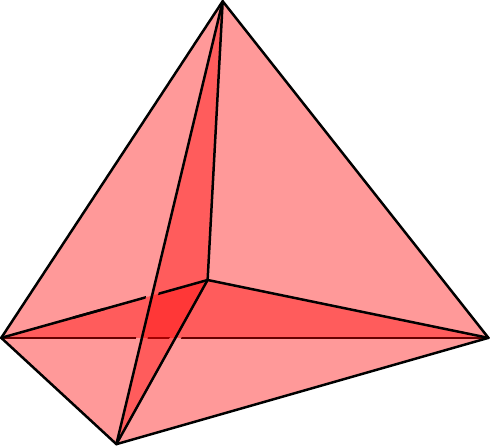}};

	% Arrows with labels
	\begin{scope}[very thick, line cap=round, -{Stealth}]
	\draw ($(Before.east)+(0,0.3)$)
		-- ($(After.west)+(0,0.3)$)
		node[midway, above, inner sep=3pt] {1-4};
	\draw ($(After.west)+(0,-0.3)$)
		-- ($(Before.east)+(0,-0.3)$)
		node[midway, below, inner sep=3pt] {4-1};
	\end{scope}
	\end{tikzpicture}
\caption{The 1-4 and 4-1 moves.}
\label{fig:1-4Move}
\end{figure}

Notice that the 2-3 and 3-2 moves do not change the number of vertices,
and also that these moves can only be performed on a triangulation with at least two tetrahedra.
In fact, if we restrict our attention to one-vertex triangulations with at least two tetrahedra,
then only 2-3 and 3-2 moves are necessary to connect any two such triangulations of the same $3$-manifold;
this was proven independently by Matveev~\cite[p.~29]{Matveev2007}
and Piergallini~\cite{Piergallini1988}.
Moreover, as we discuss in Section~\ref{subsec:moves}, the Matveev-Piergallini theorem extends
to more general settings than just one-vertex triangulations of compact
$3$-manifolds~\cite{BenedettiPetronio1995,BenedettiPetronio1997,Amendola2005}.

With this in mind, recall that we often rely on invariants to give a ``no'' answer to the homeomorphism problem.
The Matveev-Piergallini theorem provides a natural way to prove that a particular object associated to
a triangulation is actually an invariant: prove that this object is preserved by 2-3 moves.
%One relatively recent example of this is due to Garoufalidis and Kashaev~\cite{GaroufalidisKashaev2019},
%who give a particular meromorphic function that is preserved by 2-3 moves.
Indeed, a number of $3$-manifold invariants have been established using 2-3 moves and/or similar moves on triangulations;
for example, see~\cite{GaroufalidisKashaev2019,TuraevViro1992}.

On the other hand, one way to give a ``yes'' answer to the homeomorphism problem is
to find a sequence of elementary moves that connects the two input triangulations.
For well-structured $3$-manifolds such as Seifert fibre spaces or cusped hyperbolic manifolds,
this is not the most efficient method, because we have access to specialised
recognition algorithms~\cite{Matveev2007,MatveevRecognizer,Weeks1993}.
However, for arbitrary $3$-manifolds, performing a computationally expensive search for
a sequence of elementary moves is currently the best technique available.

Beyond Pachner moves, it is often useful to include other elementary moves in our toolbox.
In particular, the following moves are quite widely-used:
\begin{itemize}
\item the \textbf{0-2 move}, which ``fattens up'' a pair of adjacent triangles into
a pair of tetrahedra attached around an edge; and
\item the \textbf{2-0 move}, which is the inverse of the 2-0 move.
\end{itemize}
These moves are illustrated in Figure~\ref{fig:0-2Move}.
The 0-2 and 2-0 moves can be found under various names in other sources~\cite{Amendola2005, Matveev1987, Matveev2007, Piergallini1988, RST2019}.
The 0-2 move always preserves the underlying $3$-manifold.
The same is true for the 2-0 move provided some simple conditions are satsified;
we discuss this in more detail at the end of Section~\ref{subsec:moves}.

\begin{figure}[htbp]
\centering
	\begin{tikzpicture}
	% Subpictures
	\node[inner sep=0pt] (Before) at (0,0)
		{\includegraphics[scale=0.7]{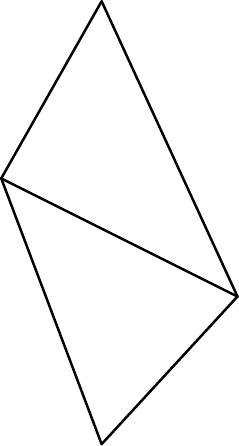}};
	\node[inner sep=0pt] (After) at (5,0)
		{\includegraphics[scale=0.7]{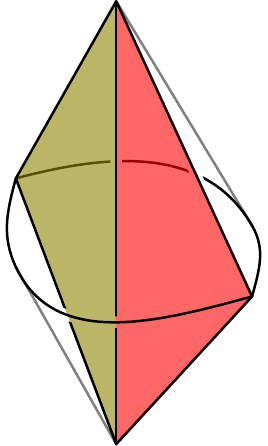}};

	% Arrows with labels
	\begin{scope}[very thick, line cap=round, -{Stealth}]
	\draw ($(Before.east)+(0,0.3)$)
		-- ($(After.west)+(0,0.3)$)
		node[midway, above, inner sep=3pt] {0-2};
	\draw ($(After.west)+(0,-0.3)$)
		-- ($(Before.east)+(0,-0.3)$)
		node[midway, below, inner sep=3pt] {2-0};
	\end{scope}
	\end{tikzpicture}
\caption{The 0-2 and 2-0 moves.}
\label{fig:0-2Move}
\end{figure}

\subsection{Unimodal sequences}\label{subsec:introMonotonic}

As we saw above, elementary moves---in particular, connectivity results
like the Matveev-Piergallini theorem---play an important role in both
theoretical and practical applications in computational $3$-manifold topology.
However, a drawback of the Matveev-Piergallini theorem is that it only says that \emph{some} sequence
of 2-3 and 3-2 moves exists, and says nothing about what this sequence actually looks like.
From a theoretical perspective, knowing more about the structure of the sequence could simplify proofs;
from a practical perspective, we could potentially exploit any extra structure to perform a targeted search
for a sequence of elementary moves, instead of relying on a brute-force search.

In this paper, we study sequences of moves that are \textbf{unimodal}
in the sense that they break up into two parts:
\begin{itemize}
\item first, a \textbf{(monotonic) ascent} which consists only of moves that increase the number of tetrahedra
(such as 2-3 or 0-2 moves); and
\item second, a \textbf{(monotonic) descent} which consists only of moves that decrease the number of tetrahedra
(such as 3-2 or 2-0 moves).
\end{itemize}
A very natural question to ask is whether two one-vertex triangulations of the same $3$-manifold are always connected by a unimodal sequence of
2-3 and 3-2 moves (that is, a sequence consisting of an ascent via 2-3 moves followed by a descent via 3-2 moves);
later on, we give a more general and more precise statement in Conjecture~\ref{conj:monotonic}.
In Section~\ref{sec:practical}, we give some experimental evidence that Conjecture~\ref{conj:monotonic} may be true.

Our main result in this paper (Theorem~\ref{thm:semiMono}) is that if we restrict our attention to one-vertex triangulations with at least two tetrahedra,
then any two such triangulations of the same $3$-manifold are connected by a unimodal sequence of 2-3 and 2-0 moves;
that is, the descent uses 2-0 moves, rather than 3-2 moves.
Alternatively, we could follow the sequence backwards to obtain a unimodal sequence of 0-2 and 3-2 moves.
(Actually, just as the Matveev-Piergallini theorem extends to more general settings,
our main result also extends to these more general settings.)
Most of Section~\ref{sec:semiMono} is devoted to proving Theorem~\ref{thm:semiMono};
we also briefly discuss how our proof technique could potentially extend to a proof of Conjecture~\ref{conj:monotonic}.

As mentioned earlier, we study the characteristics of unimodal sequences experimentally in Section~\ref{sec:practical}.
Specifically, we implemented algorithms to search for both unstructured and unimodal
sequences that connect one-vertex triangulations of the same $3$-manifold;
we tested these algorithms on 2984 distinct $3$-manifolds, which produced over 200~million intermediate triangulations.
The results suggest that, despite its poorer theoretical bounds, a blind search for an unstructured sequence gives better practical
performance due to requiring fewer additional tetrahedra in intermediate triangulations.

We finish this introduction by mentioning a couple of related results that have previously appeared in the literature:
\begin{itemize}
\item There has been some previous attention given to unimodal sequences, though using different terminology.
Specifically, in a 1999 paper~\cite{Makovetskii1999}, Makovetskii studied unimodal sequences involving 2-3, 0-2, 3-2 and 2-0 moves
(where the moves could occur in any order, as long as all the 2-3 and 0-2 moves occur before all the 3-2 and 2-0 moves).
Our proof of Theorem~\ref{thm:semiMono} was conducted independently, before we became aware of Makovetskii's work.
Makovetskii's paper is rather terse and contains some omissions, but a high-level comparison indicates
that there are some similarities between our proof strategy and that of Makovetskii.
As we discuss in Section~\ref{subsec:arch}, our proof is heavily influenced by work of Matveev~\cite{Matveev2007};
this common influence is the source of at least some of the similarity with Makovetskii's work.
In any case, we were able to contribute enough new insight to obtain a substantially stronger result.
\item A similar philosophy has previously been applied to Reidemeister moves for links in the $3$-sphere.
Specifically, Coward showed~\cite{Coward2006} that any two diagrams of a link can be
related by a sequence of Reidemeister moves sorted in the following order:
	\begin{enumerate}[label={(\arabic*)}]
	\item Reidemeister I moves that increase the number of crossings.
	\item Reidemeister II moves that increase the number of crossings.
	\item Reidemeister III moves (which preserve the number of crossings).
	\item Reidemeister II moves that decrease the number of crossings.
	\end{enumerate}
\end{itemize}

\section{Preliminaries}\label{sec:prelims}
\subsection{Triangulations and special spines}\label{subsec:triAndSpines}
Triangulations and special spines give two dual ways to represent a $3$-manifold.
Here, we define these dual notions, and describe how to convert between them.
Much of our discussion is based on that in \cite{Matveev2007} and \cite{RST2019}, though we omit the details of some proofs.

A \textbf{(generalised) triangulation} $\mathcal{T}$ is a collection of $n$
tetrahedra, such that each of the $4n$ triangular faces is affinely
identified with one of the other triangular faces; intuitively, the $n$
tetrahedra have been ``glued together'' along their faces to form a
$3$-dimensional complex.
Each pair of identified faces is referred to as a single \textbf{face}
of the triangulation. Since the face identifications cause many
tetrahedron edges to be identified and many tetrahedron vertices to be
identified, we also define an \textbf{edge} of the triangulation to be
an equivalence class of identified tetrahedron edges, and a
\textbf{vertex} of the triangulation to be an equivalence class of
identified tetrahedron vertices.
Here we allow multiple edges of the same tetrahedron to be identified, and we allow multiple vertices of the same tetrahedron to be identified.

We call $\mathcal{T}$ a \textbf{3-manifold triangulation} if its underlying topological space is a (compact) $3$-manifold.
In the definition we gave above, we did not allow any faces of tetrahedra to be left unglued;
thus, for the purposes of this paper, a $3$-manifold triangulation must represent a \textbf{closed} $3$-manifold (i.e., a compact $3$-manifold without boundary).
Also, note that a generalised triangulation $\mathcal{T}$ need not be a $3$-manifold triangulation,
because it may contain ``singularities'' at midpoints of edges and at vertices.
Depending on the face identifications that make up $\mathcal{T}$, there are two possibilities for an edge $e$.
\begin{itemize}
\item If $e$ is identified with itself in reverse, then every small
neighbourhood of the midpoint of $e$ will be bounded by a projective
plane. In this case, the midpoint of $e$ is a singularity, and
$\mathcal{T}$ is not a $3$-manifold triangulation.
\item Otherwise, every interior point of $e$ will have a small neighbourhood
bounded by a sphere, in which case the edge does not produce a singularity.
\end{itemize}
Throughout this paper, we will tacitly assume that no edge is identified with itself in reverse, so that singularities only occur at the vertices of a triangulation.

To see what can happen at a vertex $v$, imagine introducing a small
triangle near each tetrahedron vertex that is incident with $v$. As
illustrated in Figure \ref{fig:vertexLink}, we can glue all these
triangles together to form a closed surface called the \textbf{link} of
$v$. If the link is a $2$-sphere, then $v$ is not a singularity, and we call
$v$ an \textbf{internal vertex}; otherwise, we call $v$ an \textbf{ideal
vertex}. A triangulation that satisfies the edge condition above
is a $3$-manifold triangulation if and only if every vertex is internal.

\begin{figure}[htbp]
\centering
\includegraphics[scale=1]{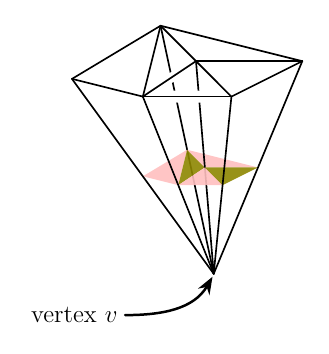}
\caption{A fragment of the link of $v$, built from triangles.}
\label{fig:vertexLink}
\end{figure}

The motivation for working with these generalised $3$-manifold triangulations is that they tend
to be smaller, and hence better suited to computation, than simplicial complexes.
In particular, every $3$-manifold admits a \textbf{one-vertex triangulation}:
an $n$-tetrahedron triangulation in which all $4n$ tetrahedron vertices are identified to become a single vertex.

Generalised triangulations also give us the flexibility to work with \textbf{ideal triangulations}: triangulations that have one or more ideal vertices.
Although the underlying topological space of an ideal triangulation $\mathcal{T}$ is not a $3$-manifold, we can recover a \textbf{bounded} $3$-manifold (i.e., a compact $3$-manifold with non-empty boundary) $\mathcal{M}$ by deleting a small open regular neighbourhood of the ideal vertices of $\mathcal{T}$;
we say that $\mathcal{T}$ is an \textbf{ideal triangulation of} $\mathcal{M}$, and that $\mathcal{M}$ is obtained from $\mathcal{T}$ by \textbf{truncating} the ideal vertices.
(Alternatively, it is also common to recover a \emph{non-compact} $3$-manifold by deleting, rather than truncating, the ideal vertices;
we do not use this idea in this paper.)
The idea for such ideal triangulations originated with Thurston's famous two-tetrahedron ideal triangulation of
the figure-eight knot complement (see~\cite{Thurston1978} or~\cite[Example~1.4.8]{Thurston1997}).
Ideal triangulations are often useful for computation because they typically require
fewer tetrahedra than non-ideal triangulations of the same bounded $3$-manifold
(such non-ideal triangulations can be obtained by allowing some faces of some tetrahedra to be left unglued).

\begin{remark}\label{rmk:noS2Bdry}
Since we do not truncate the internal vertices, the idea just described never gives a $3$-manifold with $2$-sphere boundary components.
For this reason, we will often find it convenient to restrict our attention to $3$-manifolds with no $2$-sphere boundary components.
\end{remark}

As mentioned earlier, we can also represent
$3$-manifolds using the dual notion of a special spine, which is
essentially a $2$-dimensional complex that captures all the topological
information in a $3$-dimensional manifold.
Before defining special spines precisely, it helps to get an intuitive feel for these objects by
reviewing the procedure for turning a triangulation into its dual special spine.

We start by defining a $2$-dimensional configuration called a \textbf{butterfly} that dualises the combinatorics of a tetrahedron.
Specifically, a butterfly consists of four arcs that meet at a central vertex, with a $2$-dimensional ``wing'' attached to
each of the six possible pairs of arcs, as shown in Figure~\ref{subfig:tetButterfly};
the four arcs are dual to the four triangular faces of a tetrahedron, and the six wings are dual to the six edges of a tetrahedron.
To turn a triangulation into its dual spine, replace each tetrahedron with its dual butterfly;
as illustrated in Figure~\ref{subfig:adjButterflies}, the face-gluings of the triangulation induce a natural way to attach all the dual butterflies together to form the dual spine.

\begin{figure}[htbp]
\centering
	\begin{subfigure}[t]{0.375\textwidth}
	\centering
		\includegraphics[scale=0.7]{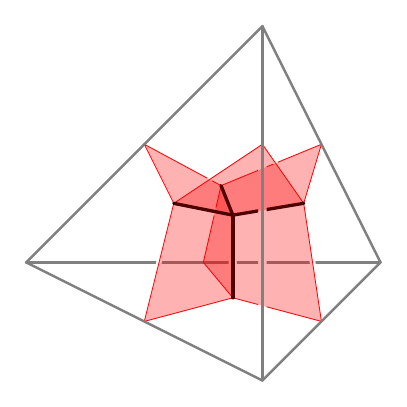}
	\caption{The dual butterfly inside a tetrahedron.}
	\label{subfig:tetButterfly}
	\end{subfigure}
	\hfill
	\begin{subfigure}[t]{0.6\textwidth}
	\centering
		\includegraphics[scale=0.7]{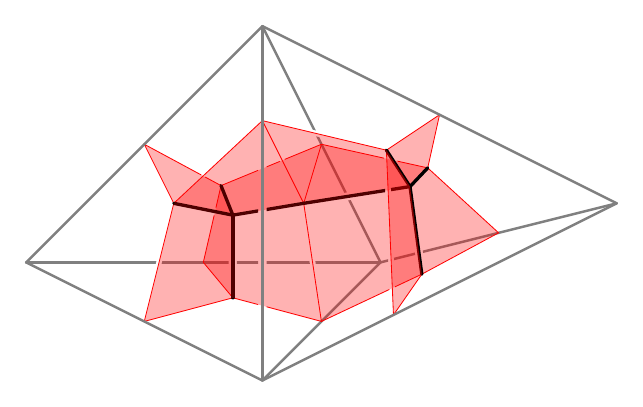}
	\caption{The dual butterflies inside a pair of adjacent tetrahedra.}
	\label{subfig:adjButterflies}
	\end{subfigure}
\caption{Converting a triangulation into its dual spine.}
\label{fig:dualSpine}
\end{figure}

It is not difficult to see that the dual spines constructed in this way are always a type of $2$-dimensional complex
called a \emph{special polyhedron}, which we characterise via the following two definitions:

\begin{definition}\label{def:simplePoly}
A \textbf{simple polyhedron} is a finite CW-complex $P$ such that each point $x$ in $P$ is of one of the following three types:
	\begin{description}[font=\bfseries\itshape]
	\item[Non-singular point.]
	Such a point $x$ has a neighbourhood in $P$ homeomorphic to a disc, as shown in Figure~\ref{subfig:nonSingular}.
	\item[Triple point.]
	Such a point $x$ has a neighbourhood in $P$ homeomorphic to three half-discs branching out from
	a common diameter passing through $x$, as shown in Figure~\ref{subfig:triplePoint}.
	\item[Vertex.]
	Such a point $x$ has a neighbourhood homeomorphic to a butterfly, with $x$ lying at
	the central vertex of this butterfly, as shown in Figure~\ref{subfig:vertexPoint}.
	\end{description}
The triple points and vertices are collectively called \textbf{singular points} of $P$;
the union of all the singular points, denoted $S_P$, is called the \textbf{singular graph} of $P$.
We write $P^{(0)}$ for the union of all vertices, $P^{(1)}$ for the union of
all triple points, and $P^{(2)}$ for the union of all non-singular points.
We refer to the components of $P^{(1)}$ as \textbf{edges}, and the components of $P^{(2)}$ as \textbf{faces}.
\end{definition}

\begin{figure}[htbp]
\centering
	\begin{subfigure}[b]{0.3\textwidth}
	\centering
	\includegraphics[scale=0.65]{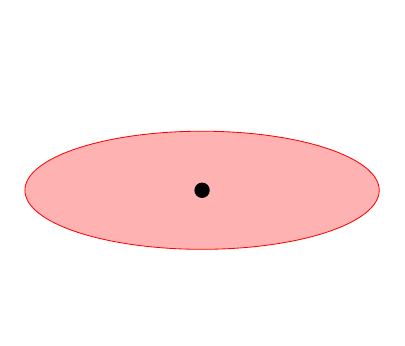}
	\caption{Non-singular point.}
	\label{subfig:nonSingular}
	\end{subfigure}
\hfill
	\begin{subfigure}[b]{0.3\textwidth}
	\centering
	\includegraphics[scale=0.65]{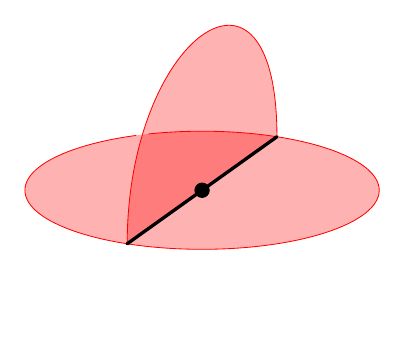}
	\caption{Triple point.}
	\label{subfig:triplePoint}
	\end{subfigure}
\hfill
	\begin{subfigure}[b]{0.3\textwidth}
	\centering
	\includegraphics[scale=0.65]{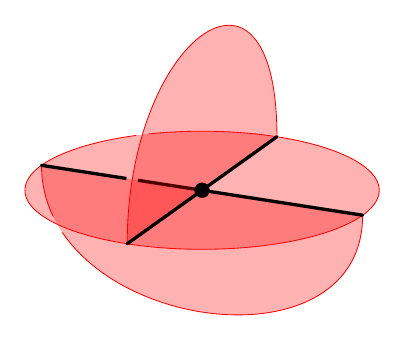}
	\caption{Vertex.}
	\label{subfig:vertexPoint}
	\end{subfigure}
\caption{The three types of points allowed in a simple polyhedron.}
\label{fig:typesOfPoints}
\end{figure}

\begin{definition}\label{def:specialPoly}
A \textbf{special polyhedron} is a simple polyhedron $P$ such that:
	\begin{itemize}[nosep]
	\item $P$ has at least one vertex (in other words, $P^{(0)}$ is non-empty);
	\item $P$ has at least one edge, and these edges are all $1$-cells
	(in other words, $P^{(1)}$ is non-empty and does not contain any closed loops); and
	\item the faces of $P$ are all discs.
	\qedhere
	\end{itemize}
\end{definition}

As we hinted earlier, a special \emph{spine} of a $3$-manifold $\mathcal{M}$ is a special polyhedron that is embedded
in $\mathcal{M}$ in such a way that it captures all the topological information in $\mathcal{M}$.
The following is a precise characterisation of this idea:

\begin{definition}\label{def:specialSpine}
Let $\mathcal{M}$ be a compact $3$-manifold, and let $Z$ denote a disjoint union of finitely many small open balls
inside $\mathcal{M}$ such that $\mathcal{M}-Z$ is a bounded $3$-manifold;
thus, $Z$ may be empty if $\mathcal{M}$ is bounded, but if $\mathcal{M}$ is closed then $Z$ must include at least one ball.
A special polyhedron $P$ is a \textbf{special spine} for $\mathcal{M}$ if, for some suitable choice of $Z$,
this polyhedron $P$ is tamely embedded in the interior of $\mathcal{M}-Z$ so that
a small closed regular neighbourhood of $P$ is homeomorphic to $\mathcal{M}-Z$ itself.
\end{definition}

Consider a compact $3$-manifold $\mathcal{M}$ with no $2$-sphere boundary components (recall Remark~\ref{rmk:noS2Bdry}).
If $\mathcal{M}$ is closed, let $\mathcal{T}$ be a corresponding $3$-manifold triangulation;
otherwise, if $\mathcal{M}$ is bounded, let $\mathcal{T}$ be an ideal triangulation of $\mathcal{M}$.
Following the construction described earlier, take $P$ to be the dual spine of $\mathcal{T}$.
To see that $P$ is indeed a special spine for $\mathcal{M}$, observe that taking a small open regular neighbourhood of
the internal vertices of $\mathcal{T}$ gives precisely the required choice $Z$ of open balls in $\mathcal{M}$.
Going in the other direction, a special spine $P$ also determines a unique corresponding
dual triangulation (see~\cite[pp.~10--13]{Matveev2007} for a proof of this);
this justifies the use of the term ``dual'' to describe this relationship between triangulations and special spines.

\subsection{Moves on triangulations and special spines}\label{subsec:moves}

As mentioned in Section~\ref{subsec:introMoves}, the 2-3 and 3-2 moves are local moves that modify a triangulation without changing its topology.
Recall also that, modulo some mild necessary conditions, two homeomorphic triangulations are always connected by a sequence of 2-3 and 3-2 moves.
This was first proven in the special case of one-vertex triangulations of closed $3$-manifolds:

\begin{theorem}[Matveev~\cite{Matveev2007} and Piergallini~\cite{Piergallini1988}, independently]\label{thm:MatveevPiergallini}
Let $\mathcal{M}$ be a closed $3$-manifold.
Let $\mathcal{T}$ and $\mathcal{U}$ be two one-vertex triangulations of $\mathcal{M}$, each with at least two tetrahedra.
Then $\mathcal{T}$ is connected to $\mathcal{U}$ by a finite sequence of 2-3 and 3-2 moves.
\end{theorem}

This was extended to triangulations of closed $3$-manifolds with any number of vertices by Benedetti and Petronio~\cite{BenedettiPetronio1995,BenedettiPetronio1997},
and subsequently also extended to ideal triangulations of bounded $3$-manifolds by Amendola~\cite{Amendola2005}.
To state the most general version of Theorem~\ref{thm:MatveevPiergallini}, and also to state our main result in this paper later on, we use the following notation:

\begin{notation}\label{notn:monotonic}
In Theorem~\ref{thm:Amendola}, Conjecture~\ref{conj:monotonic} and Theorem~\ref{thm:semiMono},
let $\mathcal{M}$ denote a compact $3$-manifold with no $2$-sphere boundary components (recall Remark~\ref{rmk:noS2Bdry}).
If $\mathcal{M}$ is closed, let $\mathcal{T}$ and $\mathcal{U}$ denote $3$-manifold triangulations of $\mathcal{M}$;
otherwise, if $\mathcal{M}$ is bounded, let $\mathcal{T}$ and $\mathcal{U}$ denote ideal triangulations of $\mathcal{M}$.
Suppose that $\mathcal{T}$ and $\mathcal{U}$ each have at least two tetrahedra,
and also that these two triangulations have the same number (possibly zero) of internal vertices.
\end{notation}

\begin{theorem}[Amendola \cite{Amendola2005}]\label{thm:Amendola}
For any two triangulations $\mathcal{T}$ and $\mathcal{U}$ of a $3$-manifold $\mathcal{M}$, subject to Notation~\ref{notn:monotonic},
there is a finite sequence of 2-3 and 3-2 moves relating $\mathcal{T}$ to $\mathcal{U}$.
\end{theorem}

Theorems~\ref{thm:MatveevPiergallini} and~\ref{thm:Amendola} were originally stated and proved in the dual setting of special spines;
a proof of Theorem~\ref{thm:Amendola} can also be found in \cite{RST2019}.
One advantage of working with special spines is that it is often easier to visualise long sequences of moves in this setting.
With this in mind, we devote the remainder of this section to describing 2-3, 3-2, 0-2 and 2-0 moves on special spines.

Recall that in a triangulation, a 2-3 move takes a pair of tetrahedra that are joined along
a triangular face, and replaces them with three tetrahedra attached around a common edge.
In the dual special spine, the corresponding 2-3 move takes a pair of vertices that are
connected by an edge, and replaces them with three vertices that form the vertices of a triangular face.
Figure~\ref{fig:move23} illustrates the 2-3 and 3-2 moves in a way that emphasises the duality between the triangulation and special spine settings.

\begin{figure}[htbp]
\centering
	\begin{subfigure}[t]{0.45\textwidth}
	\centering
		\begin{tikzpicture}
		% Subpictures
		\node[inner xsep=0pt, inner ysep=0pt] (Before) at (0,0)
			{\includegraphics[scale=0.55]{Pachner2-3Before.pdf}};
		\node[inner xsep=0pt, inner ysep=0pt] (After) at (4.5,0)
			{\includegraphics[scale=0.55]{Pachner2-3After.pdf}};

		% Arrows with labels
		\begin{scope}[thick, line cap=round, -{Stealth}]
		\draw ($(Before.east)+(0,0.3)$)
			-- ($(After.west)+(0,0.3)$)
			node[midway, above, inner sep=3pt] {2-3};
		\draw ($(After.west)+(0,-0.3)$)
			-- ($(Before.east)+(0,-0.3)$)
			node[midway, below, inner sep=3pt] {3-2};
		\end{scope}
		\end{tikzpicture}
	\caption{Version for triangulations.}
	\label{subfig:tri23}
	\end{subfigure}
	\hfill
	\begin{subfigure}[t]{0.45\textwidth}
	\centering
		\begin{tikzpicture}
		% Subpictures
		\node[inner xsep=2pt, inner ysep=0pt] (Before) at (0,0)
			{\includegraphics[scale=0.6]{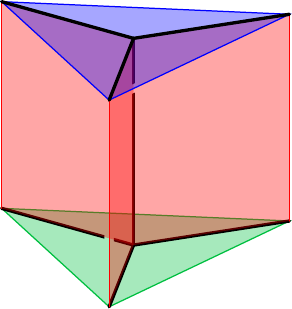}};
		\node[inner xsep=2pt, inner ysep=0pt] (After) at (4.5,0)
			{\includegraphics[scale=0.6]{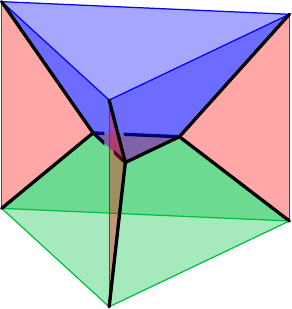}};

		% Arrows with labels
		\begin{scope}[thick, line cap=round, -{Stealth}]
		\draw ($(Before.east)+(0,0.3)$)
			-- ($(After.west)+(0,0.3)$)
			node[midway, above, inner sep=3pt] {2-3};
		\draw ($(After.west)+(0,-0.3)$)
			-- ($(Before.east)+(0,-0.3)$)
			node[midway, below, inner sep=3pt] {3-2};
		\end{scope}
		\end{tikzpicture}
	\caption{Version for special spines.}
	\label{subfig:spine23}
	\end{subfigure}
\caption{Dual pictures for the 2-3 and 3-2 moves.}
\label{fig:move23}
\end{figure}

For our purposes, it will be helpful to use a different visualisation of 2-3 moves in the special spine setting.
Specifically, let $P$ be a special spine for a $3$-manifold $\mathcal{M}$, and consider an edge $e$ that connects two distinct vertices in $P$.
Imagine taking a sheet attached to one endpoint of $e$, and dragging this sheet along $e$ until
it crosses the other endpoint, as shown in Figure~\ref{subfig:spine23Generic};
up to ambient isotopy of $P$ in $\mathcal{M}$, this operation is equivalent to performing a 2-3 move on the vertices that form the endpoints of $e$.
When visualising 2-3 moves in this way, we will often simplify our drawings by omitting the sheet that
gets dragged along $e$, as shown in Figure~\ref{subfig:spine23Simplified}.

\begin{figure}[htbp]
\raggedleft
	\begin{subfigure}[t]{0.75\textwidth}
	\centering
		\begin{tikzpicture}

		\node[inner xsep=-18pt, inner ysep=0pt] (before) at (0,0) {
			\includegraphics[scale=0.7]{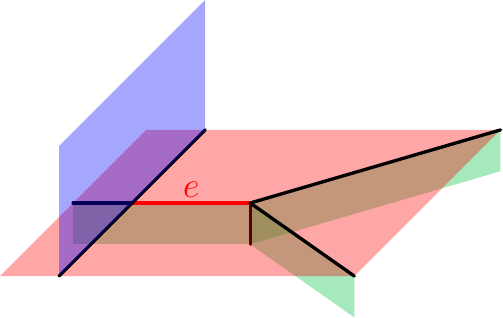}
		};
		\node[inner xsep=-18pt, inner ysep=0pt] (after) at (6.4,0) {
			\includegraphics[scale=0.7]{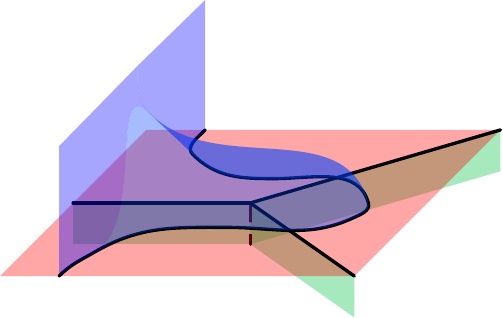}
		};
		\draw[thick, line cap=round, -Stealth] ($(before.east)+(-0.16,-0.56)$) -- node[inner sep=1pt,above]{2-3} ($(after.west)+(0,-0.56)$);

		\end{tikzpicture}
	\caption{Performing a 2-3 move by dragging a sheet along an edge $e$.}
	\label{subfig:spine23Generic}
	\end{subfigure}
	\hfill
	\bigskip
	\hfill
	\begin{subfigure}[t]{0.75\textwidth}
	\centering
		\begin{tikzpicture}

		\node[inner xsep=-18pt, inner ysep=0pt] (before) at (0,0) {
			\includegraphics[scale=0.7]{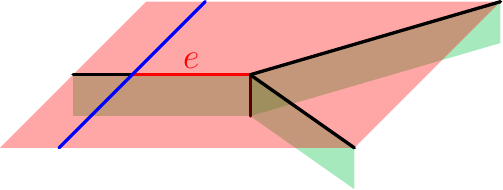}
		};
		\node[inner xsep=-18pt, inner ysep=0pt] (after) at (6.4,0) {
			\includegraphics[scale=0.7]{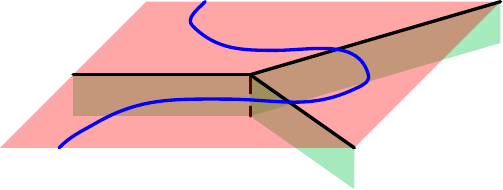}
		};
		\draw[thick, line cap=round, -Stealth] ($(before.east)+(-0.16,0.16)$) -- node[inner sep=1pt,above]{2-3} ($(after.west)+(0,0.16)$);

		\end{tikzpicture}
	\caption{A simplified picture, with the sheet extending upwards from the blue line omitted.}
	\label{subfig:spine23Simplified}
	\end{subfigure}
\caption{Alternative pictures for a 2-3 move on a special spine.}
\label{fig:spine23Alt}
\end{figure}

For 0-2 moves, recall that in a triangulation, such a move takes a pair of triangular faces that share a common edge,
and ``fattens up'' these triangles into a pair of tetrahedra attached around a new edge.
In the dual special spine, the corresponding 0-2 move takes a pair of edges that are both incident to a common face,
and creates a pair of vertices that together form the vertices of a new bigon face.
Figure~\ref{fig:move02} illustrates the 0-2 and 2-0 moves in a way that emphasises the duality between the triangulation and special spine settings.

\begin{figure}[htbp]
\centering
	\begin{subfigure}[t]{0.45\textwidth}
	\centering
		\begin{tikzpicture}
		% Subpictures
		\node[inner xsep=0pt, inner ysep=0pt] (Before) at (0,0)
			{\includegraphics[scale=0.55]{Tri0-2Before.pdf}};
		\node[inner xsep=0pt, inner ysep=0pt] (After) at (4.5,0)
			{\includegraphics[scale=0.55]{Tri0-2After.pdf}};

		% Arrows with labels
		\begin{scope}[thick, line cap=round, -{Stealth}]
		\draw ($(Before.east)+(0,0.3)$)
			-- ($(After.west)+(0,0.3)$)
			node[midway, above, inner sep=3pt] {0-2};
		\draw ($(After.west)+(0,-0.3)$)
			-- ($(Before.east)+(0,-0.3)$)
			node[midway, below, inner sep=3pt] {2-0};
		\end{scope}
		\end{tikzpicture}
	\caption{Version for triangulations.}
	\label{subfig:tri02}
	\end{subfigure}
	\hfill
	\begin{subfigure}[t]{0.45\textwidth}
	\centering
		\begin{tikzpicture}
		% Subpictures
		\node[inner xsep=2pt, inner ysep=0pt] (Before) at (0,0)
			{\includegraphics[scale=0.6]{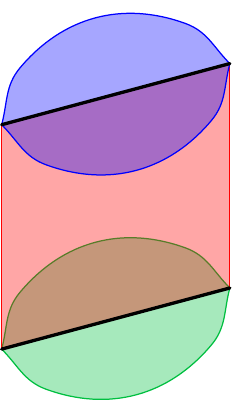}};
		\node[inner xsep=2pt, inner ysep=0pt] (After) at (4.5,0)
			{\includegraphics[scale=0.6]{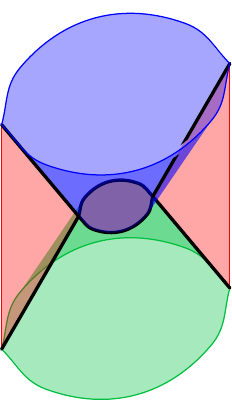}};

		% Arrows with labels
		\begin{scope}[thick, line cap=round, -{Stealth}]
		\draw ($(Before.east)+(0,0.3)$)
			-- ($(After.west)+(0,0.3)$)
			node[midway, above, inner sep=3pt] {0-2};
		\draw ($(After.west)+(0,-0.3)$)
			-- ($(Before.east)+(0,-0.3)$)
			node[midway, below, inner sep=3pt] {2-0};
		\end{scope}
		\end{tikzpicture}
	\caption{Version for special spines.}
	\label{subfig:spine02}
	\end{subfigure}
\caption{Dual pictures for the 0-2 and 2-0 moves.}
\label{fig:move02}
\end{figure}

As with 2-3 moves, we will find it helpful to use a different visualisation of 0-2 moves in the special spine setting.
In detail, let $P$ be a special spine for a $3$-manifold $\mathcal{M}$, and consider a curve
$\alpha$ embedded in a face $C$ of $P$ such that the endpoints of $\alpha$ lie on edges of $C$.
Imagine taking a sheet attached to one endpoint of $\alpha$, and dragging this sheet along $\alpha$
until it crosses the other endpoint, as shown in Figure~\ref{subfig:spine02Generic};
up to ambient isotopy of $P$ in $\mathcal{M}$, this operation is equivalent to performing a 0-2 move on the pair of edges incident to the endpoints of $\alpha$.
We will often simplify our drawings of such 0-2 moves by omitting the sheet that gets dragged
along $\alpha$, as shown in Figure~\ref{subfig:spine02Simplified}.

\begin{figure}[t]%\begin{figure}[htbp]
\raggedleft
	\begin{subfigure}[t]{0.75\textwidth}
	\centering
		\begin{tikzpicture}

		\node[inner xsep=-18pt, inner ysep=0pt] (before) at (0,0) {
			\includegraphics[scale=0.7]{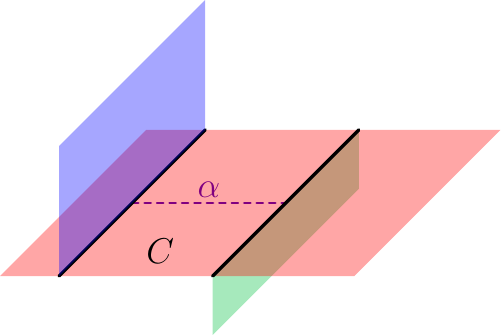}
		};
		\node[inner xsep=-18pt, inner ysep=0pt] (after) at (6.4,0) {
			\includegraphics[scale=0.7]{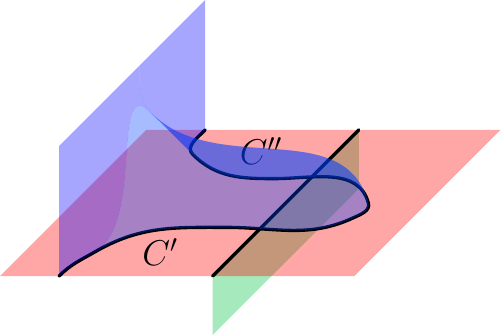}
		};
		\draw[thick, line cap=round, -Stealth] ($(before.east)+(-0.24,-0.56)$) -- node[inner sep=1pt,above]{0-2} ($(after.west)+(0,-0.56)$);

		\end{tikzpicture}
	\caption{Performing a 0-2 move by dragging a sheet along a curve $\alpha$.}
	\label{subfig:spine02Generic}
	\end{subfigure}
	\hfill
	\bigskip
	\hfill
	\begin{subfigure}[t]{0.75\textwidth}
	\centering
		\begin{tikzpicture}

		\node[inner xsep=-18pt, inner ysep=0pt] (before) at (0,0) {
			\includegraphics[scale=0.7]{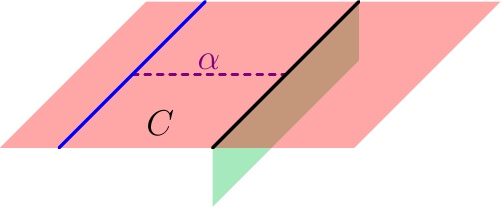}
		};
		\node[inner xsep=-18pt, inner ysep=0pt] (after) at (6.4,0) {
			\includegraphics[scale=0.7]{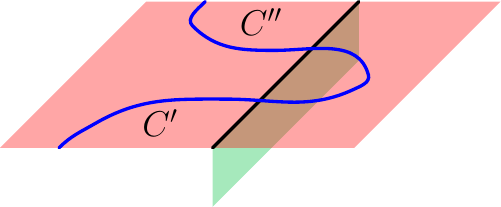}
		};
		\draw[thick, line cap=round, -Stealth] ($(before.east)+(-0.24,0.16)$) -- node[inner sep=1pt,above]{0-2} ($(after.west)+(0,0.2)$);

		\end{tikzpicture}
	\caption{A simplified picture, with the sheet extending upwards from the blue line omitted.}
	\label{subfig:spine02Simplified}
	\end{subfigure}
\caption{Alternative pictures for a 0-2 move on a special spine.}
\label{fig:spine02Alt}
\end{figure}

It is straightforward to check that performing a 0-2 move on a special spine always yields a new special spine of the same $3$-manifold.
However, performing a 2-0 move on a special spine is \emph{not} guaranteed to produce a special spine.
To see why, let $P_0$ and $P_2$ denote the spines on the left and right, respectively, of Figure~\ref{subfig:spine02Generic}.
Even if $P_2$ is special, it is possible for $P_0$ to be non-special in one of the following ways:
\begin{enumerate}[label={(\alph*)}]
\item\label{move20:coincide}
Suppose the two faces $C'$ and $C''$ coincide to form a single disc.
In this case, the face $C$ must be either an annulus or a M\"{o}bius band, so $P_0$ cannot be special.
\item\label{move20:distinct}
Suppose the two faces $C'$ and $C''$ form two distinct discs, but the two vertices shown in
Figure~\ref{subfig:spine02Generic} are the \emph{only} vertices incident to either $C'$ or $C''$.
In this case, the face $C$ must be a single disc whose boundary contains no vertices, so again $P_0$ cannot be special.
\end{enumerate}
We need to rule out these two degenerate cases to ensure that a 2-0 move preserves the property of being a special spine.

Translating this into the dual triangulation setting, we can perform a 2-0 move about an edge $e$ in a triangulation $\mathcal{T}$
(with the restriction that the result is a triangulation homeomorphic to $\mathcal{T}$) if and only if the following conditions are all satisfied:
\begin{itemize}
\item The edge $e$ has degree two (i.e., as an equivalence class of tetrahedron edges, $e$ is formed from exactly
two tetrahedron edges), and is incident to two distinct tetrahedra $\Delta'$ and $\Delta''$.
(In the dual special spine setting, this corresponds to the fact that a 2-0 move can only be performed on a bigon face that meets two distinct vertices.)
\item Let $\varepsilon'$ denote the edge of $\Delta'$ opposite $e$, and let $\varepsilon''$ denote the edge of $\Delta''$ opposite $e$.
The edges $\varepsilon'$ and $\varepsilon''$ must be distinct.
(This is dual to ruling out case~\ref{move20:coincide}.)
\item Let $f'$ and $g'$ denote the two faces of $\Delta'$ that share the common edge $\varepsilon'$,
and let $f''$ and $g''$ denote the two faces of $\Delta''$ that share the common edge $\varepsilon''$.
We cannot simultaneously have $f'$ identified to $g'$ and $f''$ identified to $g''$.
(This is dual to ruling out case~\ref{move20:distinct}.)
\end{itemize}

\section{Unimodal sequences of elementary moves}\label{sec:semiMono}

Recall from Section~\ref{subsec:introMonotonic} that we call a finite sequence of elementary moves \textbf{unimodal} if it breaks up into two parts:
\begin{itemize}
\item first, a \textbf{(monotonic) ascent} which consists only of moves that increase the number of tetrahedra
(such as 2-3 or 0-2 moves); and
\item second, a \textbf{(monotonic) descent} which consists only of moves that decrease the number of tetrahedra
(such as 3-2 or 2-0 moves).
\end{itemize}
We conjecture that the result of Amendola mentioned in Section~\ref{subsec:moves} (Theorem~\ref{thm:Amendola}) can be strengthened as follows:

\begin{restatable}{conjecture}{conjMonotonic}
\label{conj:monotonic}
For any two triangulations $\mathcal{T}$ and $\mathcal{U}$ of a $3$-manifold $\mathcal{M}$, subject to Notation~\ref{notn:monotonic},
there is a unimodal sequence of 2-3 and 3-2 moves relating $\mathcal{T}$ to $\mathcal{U}$.
\end{restatable}

This conjecture remains open.
Our main theoretical contribution in this paper is to prove a variant of Conjecture~\ref{conj:monotonic},
where the monotonic descent involves 2-0 moves rather than 3-2 moves;
see Theorem~\ref{thm:semiMono} below.
This section is mostly devoted to proving Theorem~\ref{thm:semiMono}, though at the end we also briefly discuss how our proof strategy
could potentially be adapted to prove Conjecture~\ref{conj:monotonic}.

\begin{restatable}{theorem}{thmSemiMono}
\label{thm:semiMono}
For any two triangulations $\mathcal{T}$ and $\mathcal{U}$ of a $3$-manifold $\mathcal{M}$, subject to Notation~\ref{notn:monotonic},
there is a unimodal sequence of 2-3 and 2-0 moves relating $\mathcal{T}$ to $\mathcal{U}$.
\end{restatable}

\subsection{Walls and arches}\label{subsec:arch}

Our proof of Theorem~\ref{thm:semiMono}, which we present in Section~\ref{subsec:semiMonoProof},
uses a variation of a technique from Matveev's proof of Theorem~\ref{thm:MatveevPiergallini}.
The construction at the heart of this technique is Matveev's \textbf{arch-with-membrane};
for brevity, we will simply refer to an arch-with-membrane as an \textbf{arch}.\footnote{
Matveev uses the word ``arch'' to describe a slightly different construction from the arch-with-membrane.
For our purposes, there is no risk of confusion because we will only need to work with the arch-with-membrane.}
We have two goals in this section.
\begin{itemize}
\item First, we introduce some terminology that will help us describe and work with arches.
\item Second, we use this terminology to paraphrase some key ideas from Matveev's proof, since we will be reusing these ideas in Section~\ref{subsec:semiMonoProof}.
\end{itemize}

To describe how we construct arches, we first define a sort of precursor to an arch.
For this, let $\mathcal{M}$ be a $3$-manifold with no $2$-sphere boundary components, and consider any particular special spine $P$ for $\mathcal{M}$.
A \textbf{wall} for $P$ is an embedding of the rectangle $[0,1]\times[0,1]$ into $\mathcal{M}$ such that:
\begin{itemize}
\item the intersection of the rectangle with $P$ is given by the curve $[0,1]\times\{0\}$---we call this curve the \textbf{base} of the wall;
\item the base of the wall only ever meets the singular graph $S_P$ via transverse intersections of $(0,1)\times\{0\}$ with
edges of $P$---we call each such point of intersection a \textbf{buttress} of the wall; and
\item there is at least one buttress.
\end{itemize}
The \textbf{length} $\ell$ of the wall is given by the number of buttresses minus one;
we (arbitrarily) orient the base of the wall, and label the buttresses from $0$ to $\ell$ in the order that they appear along the base.
Figure~\ref{fig:wallExample} shows an example of a wall of length $4$.

\begin{figure}[htbp]
\centering
	\includegraphics[scale=0.55]{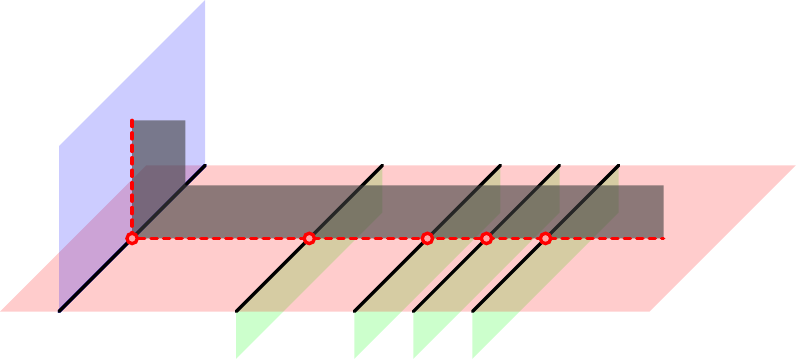}
\caption{A wall (shaded grey) of length $4$.
The base of this wall is drawn as a dashed red curve, and the five buttresses are marked by small red circles.}
\label{fig:wallExample}
\end{figure}

We first describe how to construct an arch of length $0$.
Let $b_0$ denote the \nth{0} buttress of a wall $W$, and focus on a small neighbourhood $N$ of $b_0$ inside $\mathcal{M}$.
On the edge containing $b_0$, consider two points $p$ and $p'$ that lie on either side of $b_0$;
inside $N$, let $\alpha$ denote a curve from $p$ to $p'$ that lies in the face that only meets $W$ at the buttress $b_0$.
We construct a length-$0$ arch by performing a 0-2 move along $\alpha$;
this is more easily visualised by first performing an isotopy of the spine, as illustrated in Figure~\ref{fig:arch0-2}.
We usually simplify our drawings of arches (of any length, not just of length $0$), as shown in Figure~\ref{fig:simpArch}.

\begin{figure}[htbp]
\centering
	\includegraphics[scale=0.55]{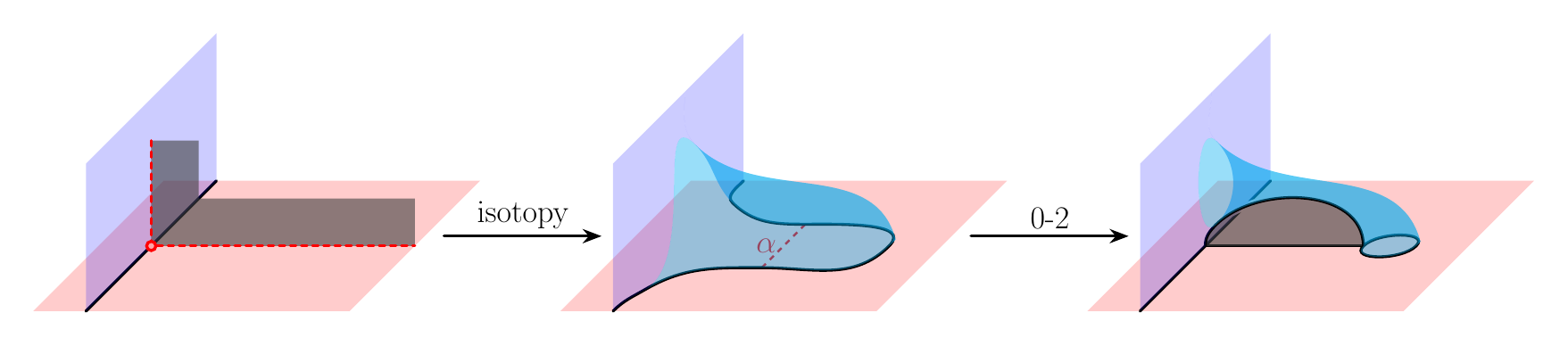}
\caption{At the \nth{0} buttress of a wall, we can create a length-$0$ arch by performing a 0-2 move along the curve $\alpha$.
Going backwards, we can destroy a length-$0$ arch by performing a 2-0 move.}
\label{fig:arch0-2}
\end{figure}

\begin{figure}[htbp]
\centering
\includegraphics[scale=1]{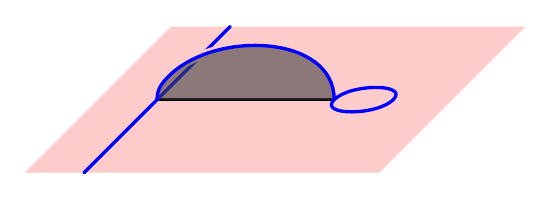}
\caption{A simplified drawing of a length-$0$ arch, with the sheet attached to the blue edges omitted.
We also simplify our drawings of positive-length arches in an analogous way.}
\label{fig:simpArch}
\end{figure}

By creating a length-$0$ arch, we introduce two new faces into our spine:
\begin{itemize}
\item a new monogon, which we call the \textbf{artificial monogon}; and
\item a new bigon (shaded grey in Figure~\ref{fig:simpArch}), which we call the \textbf{artificial membrane}.
\end{itemize}
Very shortly, we will describe how to construct arches of arbitrary positive length $\ell$;
for such arches, the artificial membrane will be an $(\ell+2)$-gon, rather than a bigon.
To go along with this terminology, we will call an edge of our spine \textbf{artificial} if it is incident to either the artificial membrane or the artificial monogon;
otherwise, we will call the edge \textbf{natural}.

To construct an arch of positive length $\ell$, start with a wall $W$ of length $\ell$.
As above, use a 0-2 move to create an arch of length $0$.
We extend this arch to length $\ell$ using the following two steps:
\begin{enumerate}[label={(\arabic*)}]
\item Perform a sequence of $\ell$ 0-2 moves along the base of the wall $W$.
\item Perform a sequence of $\ell$ 3-2 moves to turn the artificial membrane into an $(\ell+2)$-gon.
\end{enumerate}
Figure~\ref{fig:extendArch} illustrates this extension procedure for the case $\ell=4$.

\begin{figure}[htbp]
\centering
	\includegraphics[scale=0.55]{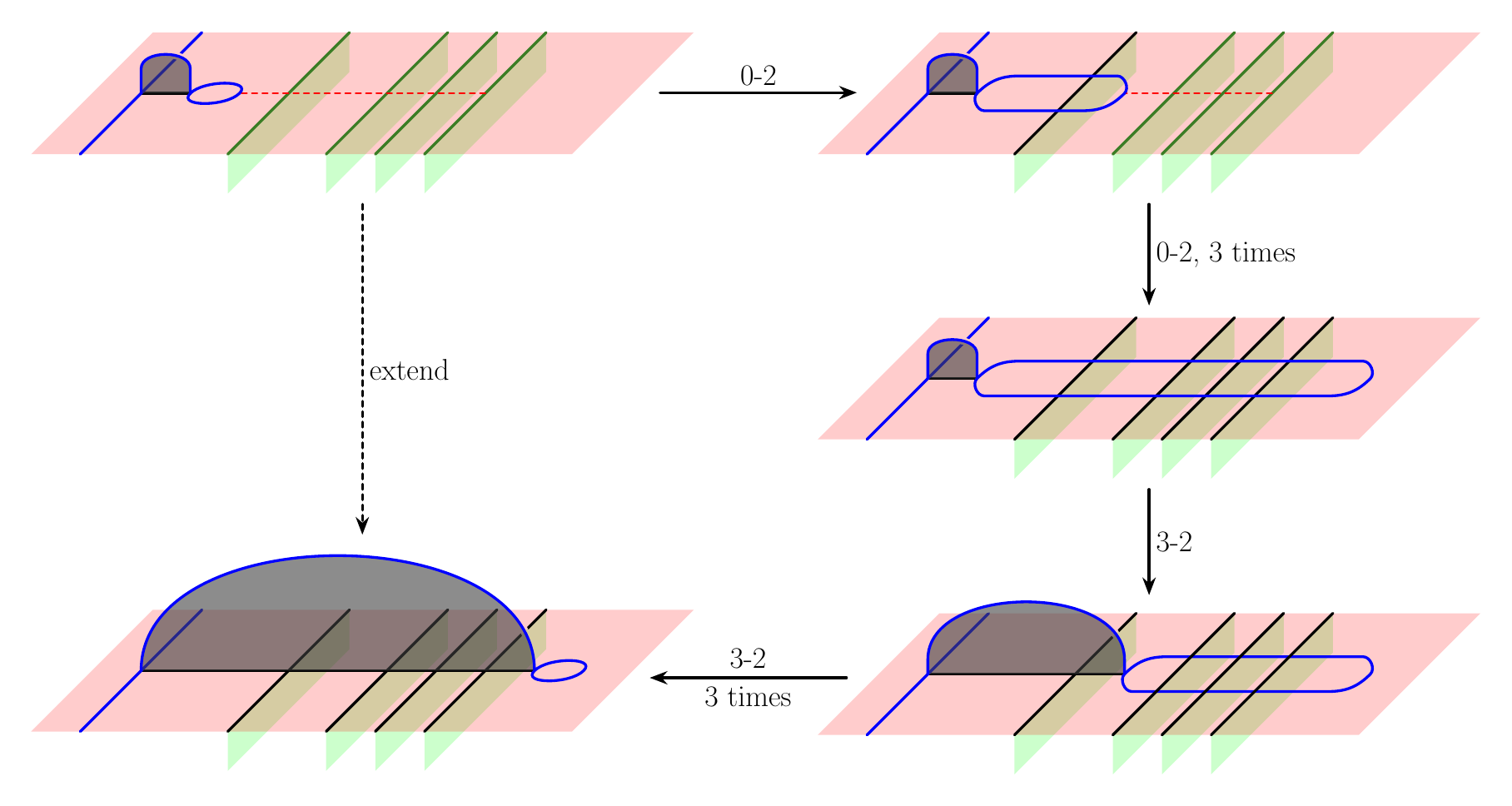}
\caption{Using a sequence of $\ell$ 0-2 moves followed by $\ell$ 3-2 moves, we can extend an arch of length $0$ to an arch of length $\ell$.
Going backwards, we can reduce the length from $\ell$ to $0$ using a sequence of $\ell$ 2-3 moves followed by $\ell$ 2-0 moves.
In this example, $\ell=4$.}
\label{fig:extendArch}
\end{figure}

\begin{observation}\label{obs:removeArch}
An arch of length $\ell$ can be \emph{removed} using a unimodal sequence consisting of $\ell$
2-3 moves followed by $(\ell+1)$ 2-0 moves. Specifically, we can do this using the following two steps:
\begin{enumerate}[nosep,label={(\arabic*)}]
\item By going backwards through the extension procedure, reduce the length to $0$ using a unimodal sequence consisting of $\ell$ 2-3 moves followed by $\ell$ 2-0 moves.
\item Use an additional 2-0 move (the inverse of the arch-creation move shown in Figure~\ref{fig:arch0-2}) to destroy the arch.
\end{enumerate}
\end{observation}

Creating an arch on a spine $P$ is useful because doing so has ``minimal impact'' on the structure of $P$.
To pin down this intuition, let $W$ denote a wall of length $\ell$ for $P$, and let $Q$ denote the spine obtained by using $W$ to add an arch of length $\ell$ to $P$.
For each vertex $v$ of $Q$, we classify this vertex into one of the following types:
\begin{itemize}
\item Call $v$ an \textbf{artificial vertex} if it is incident to \emph{both} the artificial membrane and the artificial monogon;
there is exactly one such vertex.
\item Call $v$ a \textbf{subdividing vertex} if it is incident to the artificial membrane, but \emph{not} incident to the artificial monogon;
there are exactly $\ell+1$ such vertices.
\item Call $v$ a \textbf{natural vertex} if it is incident to \emph{neither} the artificial membrane nor the artificial monogon.
\end{itemize}
Observe that there is a natural bijection between the vertices of $P$ and the \emph{natural} vertices of $Q$;
there is also a natural bijection between the buttresses of $W$ and the subdividing vertices of $Q$.

With this in mind, consider the singular graphs $S_Q$ and $S_P$.
Observe that after deleting the artificial vertex and the artificial edges from $S_Q$, we obtain a new graph $G_Q$ that is a \emph{subdivision} of $S_P$:
\begin{itemize}
\item every vertex in $S_P$ corresponds to a natural vertex in $G_Q$; and
\item every edge in $S_P$ corresponds to a ``chain'' of edges in $G_Q$ that starts at a natural vertex,
possibly passes through one or more subdividing vertices, and ends at a natural vertex.
\end{itemize}
Intuitively, if we ``ignore'' the artificial vertex and artificial edges, then the singular graphs $S_Q$ and $S_P$ can be considered ``more-or-less equivalent'';
thus, we can think of the spine $Q$ as being ``just like $P$, but with an extra arch''.

We now outline how Matveev uses arches in his proof of Theorem~\ref{thm:MatveevPiergallini} (see \cite{Matveev2007} for
more details), since a similar idea will be useful for us in Section~\ref{subsec:semiMonoProof}.
Here, we rephrase Matveev's idea in a way that is better suited to our purposes.
Prior to proving Theorem~\ref{thm:MatveevPiergallini}, Matveev establishes the following (non-trivial) facts:
\begin{itemize}
\item A 0-2 move on a special spine can be replaced by a sequence (not
necessarily unimodal) of 2-3 and/or 3-2 moves.
Going backwards, if a 2-0 move results in a special
spine, then it can be replaced by a sequence of 2-3 and/or 3-2 moves.
\item If $P$ and $Q$ are special spines of the same $3$-manifold, each
with at least two vertices, then $P$ is connected to $Q$ by a sequence
$\varSigma$ of 2-3, 3-2, 0-2 and/or 2-0 moves.
\end{itemize}
From here, Matveev proceeds by turning the sequence $\varSigma$ into a new sequence that only uses 2-3 and 3-2 moves.

Although the 0-2 moves in $\varSigma$ can always be replaced by 2-3 and 3-2 moves,
we might not be able to replace all of the 2-0 moves, because a 2-0 move could yield a spine that is not special.
Matveev circumvents this issue by creating an arch every time a 2-0 move is supposed to be performed.
To see how this works, consider a 2-0 move $\vartheta$ that turns a spine $P$ into another spine $Q$.
Instead of performing the move $\vartheta$, we can perform the 2-3 move illustrated in Figure~\ref{fig:replace2-0withArch},
which turns $P$ into a spine that is just like $Q$ except it has an extra arch of length $1$.
Note that this 2-3 move is always possible, since we must initially have two distinct vertices to perform the 2-0 move $\vartheta$.
(In \cite{Matveev2007}, Matveev actually creates the arch using a 0-2 move followed by a 3-2 move.
Our 2-3 move gives the same result, but is more useful for our purposes in Section~\ref{subsec:semiMonoProof}.)

\begin{figure}[htbp]
\centering
	\includegraphics[scale=0.7]{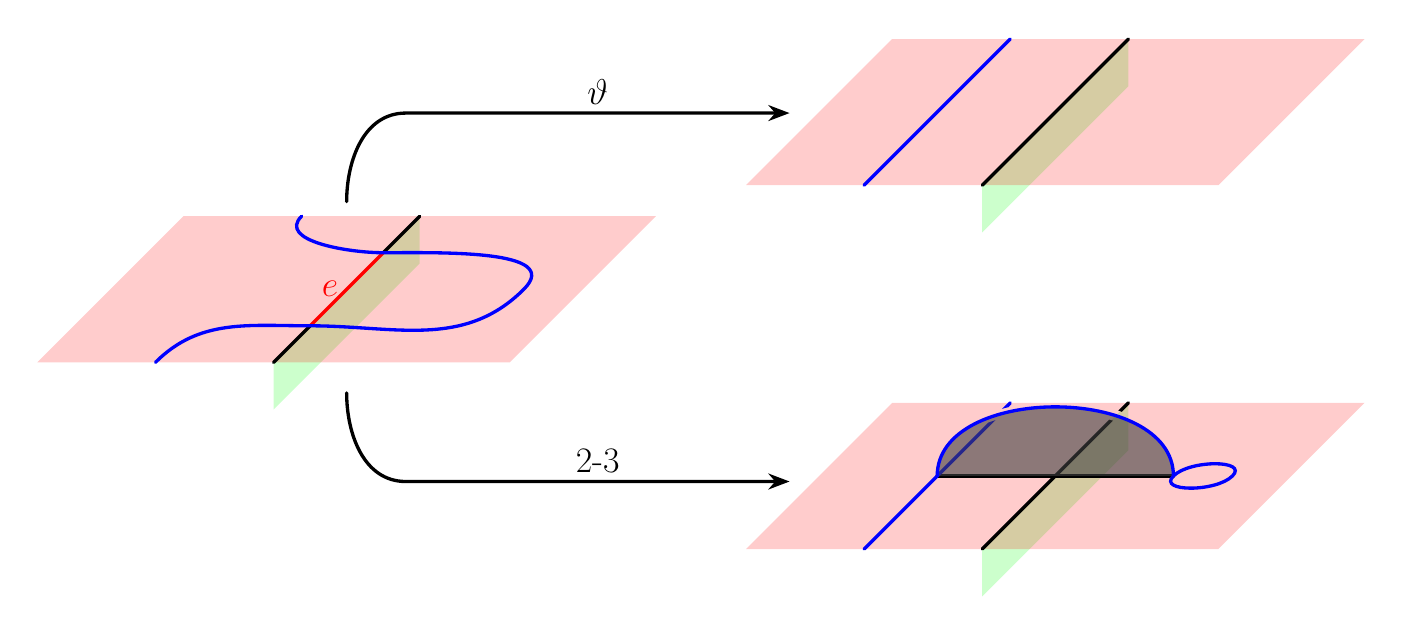}
\caption{Instead of turning $P$ into $Q$ by performing the 2-0 move $\vartheta$, we can perform a
2-3 move along the red edge $e$ to turn $P$ into a spine that looks like $Q$ with an extra arch.}
\label{fig:replace2-0withArch}
\end{figure}

As we have already mentioned, Matveev's idea is to replace every 2-0 move by creating an arch in the manner just described.
The extra arches that arise in this way can be moved around quite freely using 2-3, 3-2 and 0-2 moves,
so that they never prevent us from performing any of the 2-3, 3-2 and 0-2 moves in $\varSigma$;
the specific details are unimportant here, but it is worth noting that we rely on a slight variation of this idea in Section~\ref{subsec:semiMonoProof}.
At the end, we can use Observation~\ref{obs:removeArch} to remove all the extra arches using 2-3 and 2-0 moves;
it is not hard to see that every intermediate spine in this removal process must be special, which means that the 2-0 moves can be realised by 2-3 and 3-2 moves.
In this way, Matveev turns $\varSigma$ into a sequence consisting only of 2-3 and 3-2 moves.

\subsection{Constructing unimodal sequences}\label{subsec:semiMonoProof}

In this section, we prove Theorem~\ref{thm:semiMono} by giving a procedure for transforming a non-unimodal sequence into a unimodal one.
At the end, we also briefly discuss what would be required to adapt this strategy to prove Conjecture~\ref{conj:monotonic}.

To begin, we use the ideas discussed in Section~\ref{subsec:arch} to prove Lemma~\ref{lem:localSemiMono} below.
We then use this lemma to prove Theorem~\ref{thm:semiMono}.

\begin{lemma}\label{lem:localSemiMono}
Let $P$ and $Q$ be two special spines, and suppose $P$ is connected to $Q$ by a sequence $\gamma,\delta_1,\ldots,\delta_k$,
where $k\geqslant0$, $\gamma$ is a 3-2 move, and each $\delta_i$ is a 2-3 move.
Then $P$ is related to $Q$ by a unimodal sequence of 2-3 and 2-0 moves.
\end{lemma}

\begin{proof}
Let $Q_0$ denote the spine obtained from $P$ by applying the 3-2 move $\gamma$, and for each $i\in\{1,\ldots,k\}$ let $Q_i$
denote the spine obtained from $Q_{i-1}$ by applying the 2-3 move $\delta_i$ (so, in particular, $Q_k=Q$).
Our strategy is to show that there is a sequence $P_0,\ldots,P_k$ of spines such that:
\begin{itemize}
\item for each $i\in\{0,\ldots,k\}$, $P_i$ looks just like $Q_i$ but with an extra arch;
\item $P$ can be transformed into $P_0$ using two 2-3 moves; and
\item for each $j\in\{1,\ldots,k\}$, $P_{j-1}$ can be transformed into $P_j$ using some finite sequence of 2-3 moves.
\end{itemize}
Once we have such a sequence, we can complete the proof by using Observation~\ref{obs:removeArch} to remove the extra arch in $P_k$,
and hence transform $P_k$ into $Q$ using a unimodal sequence of 2-3 and 2-0 moves.

To construct the sequence $P_0,\ldots,P_k$, we proceed by induction on $k$.
When $k=0$, all we need to show is that we can use two 2-3 moves to turn $P$ into a new spine $P_0$ that looks like $Q_0$ with an extra arch.
In other words, at the cost of introducing an arch, we need to ``copy'' the 3-2 move $\gamma$ using two 2-3 moves;
we can do this in the manner illustrated in Figure~\ref{fig:copycatGamma}.

\begin{figure}[htbp]
\centering
	\includegraphics[scale=0.55]{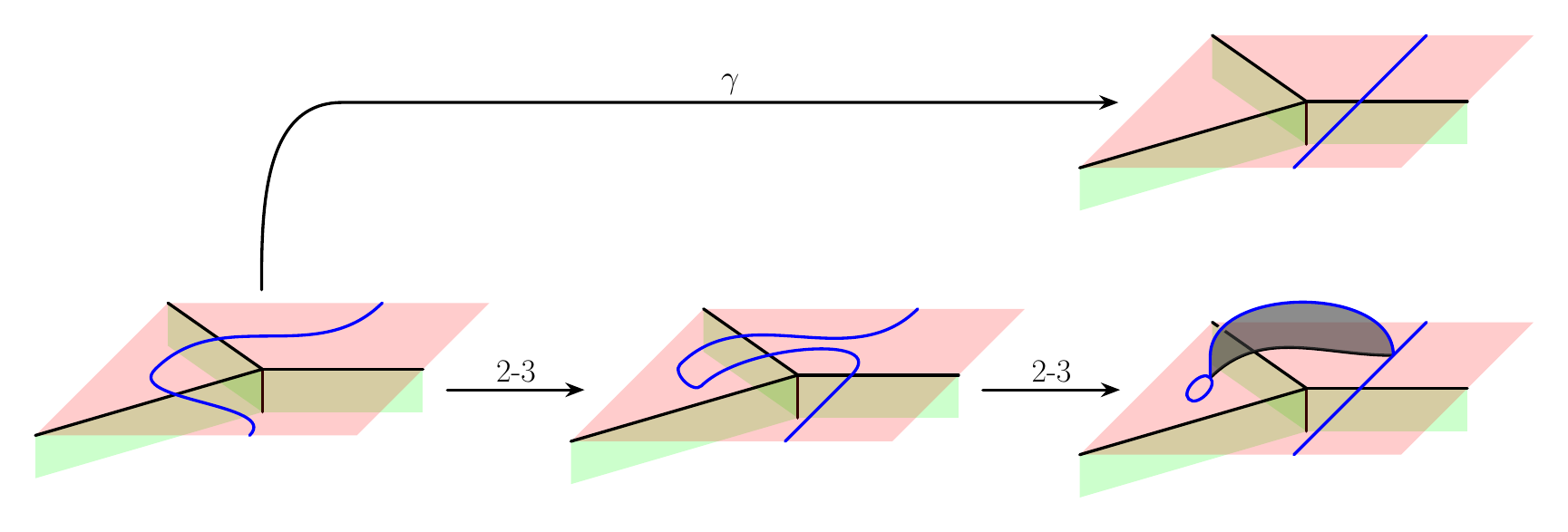}
\caption{Converting the 3-2 move $\gamma$ into 2-3 moves, at the cost of creating an extra arch.}
\label{fig:copycatGamma}
\end{figure}

Before we describe the inductive step, it is helpful to consider an explicit example with $k=1$.
Specifically, consider the 3-2 move $\gamma$ and the 2-3 move $\delta_1$ shown in Figure~\ref{fig:replacementExample}.
After ``copying'' $\gamma$ using two 2-3 moves, we would like to ``copy'' $\delta_1$.
However, $\delta_1$ is initially obstructed by the extra arch that we created when we ``copied'' $\gamma$.
To circumvent this, we simply use an additional 2-3 move to shift the arch out of the way, which then allows us to perform a 2-3 move identical to $\delta_1$.
This captures the essence of our proof:
once we have ``copied'' the initial 3-2 move $\gamma$, we can always ``copy'' the remaining 2-3 moves
$\delta_1,\ldots,\delta_k$, using additional 2-3 moves to shift the arch whenever necessary.

\begin{figure}[htbp]
\centering
	\begin{tikzpicture}

	\newcommand{\Scale}{0.42}
	\newcommand{\ShiftX}{5.6}
	\newcommand{\ShiftY}{-3.5}
	\newcommand{\ShiftArrow}{-0.6}

	% Top row
	\node[inner xsep=-10pt, inner ysep=0pt] (exaInitial) at (0,0) {
		\includegraphics[scale=\Scale]{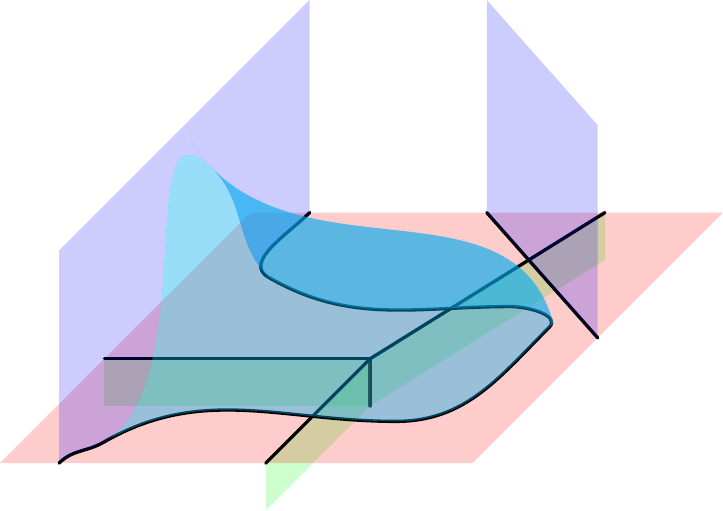}
	};
	\node[inner xsep=-10pt, inner ysep=0pt] (exaAfter32) at (\ShiftX,0) {
		\includegraphics[scale=\Scale]{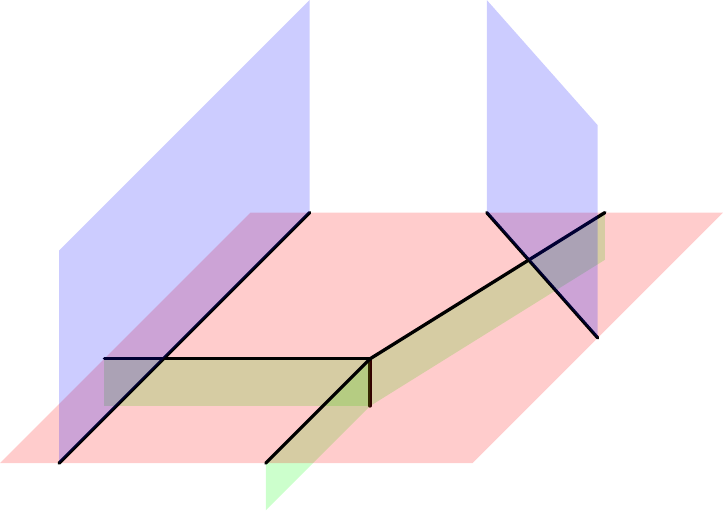}
	};
	\node[inner xsep=-10pt, inner ysep=0pt] (exaFinal) at (2*\ShiftX,0) {
		\includegraphics[scale=\Scale]{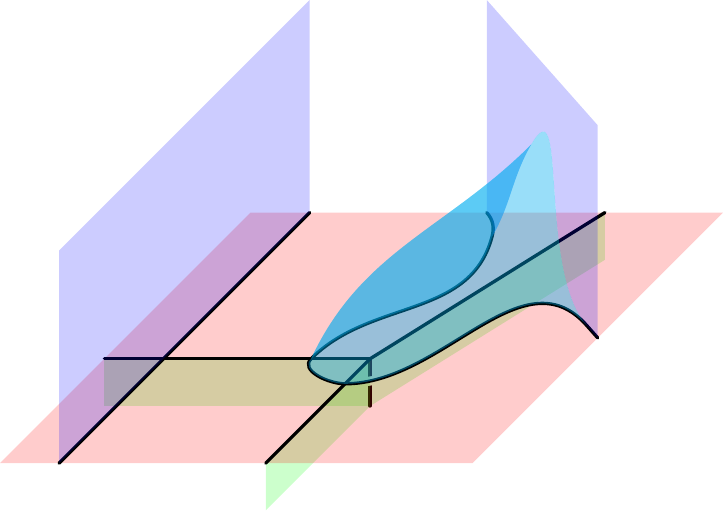}
	};

	% Middle row
	\node[inner xsep=-10pt, inner ysep=-7pt] (simInitial) at (0,\ShiftY) {
		\includegraphics[scale=\Scale]{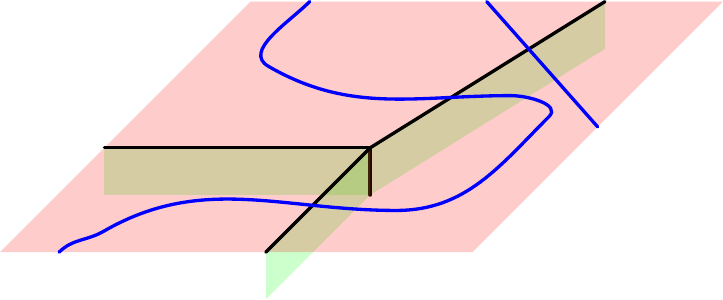}
	};
	\node[inner xsep=-10pt, inner ysep=-7pt] (simArch) at (\ShiftX,\ShiftY) {
		\includegraphics[scale=\Scale]{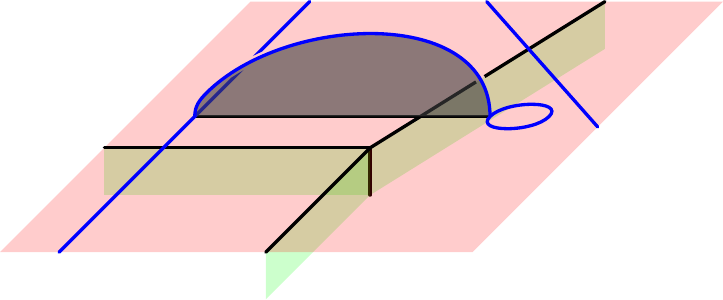}
	};
	\node[inner xsep=-10pt, inner ysep=-7pt] (simShift23) at (2*\ShiftX,\ShiftY) {
		\includegraphics[scale=\Scale]{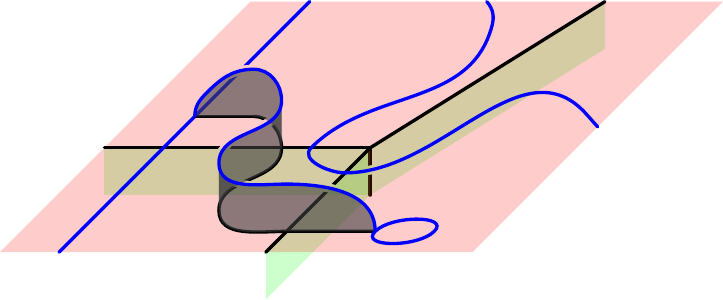}
	};

	% Bottom row
	\node[inner xsep=-10pt, inner ysep=1pt] (simShiftArch) at (1.25*\ShiftX,\ShiftY-3) {
		\includegraphics[scale=\Scale]{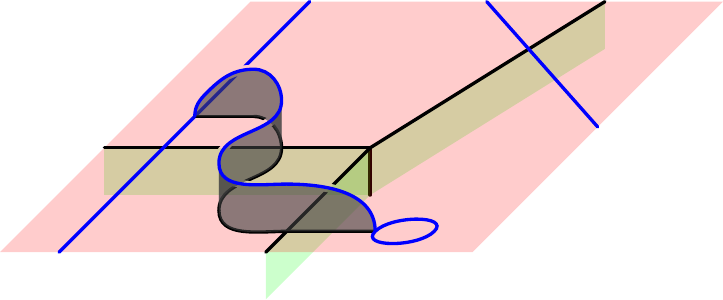}
	};

	% Top horizontal arrows
	\draw[thick, line cap=round, -Stealth] ($(exaInitial.east)+(-0.4,\ShiftArrow)$) -- node[inner sep=1pt, above]{$\gamma$} ($(exaAfter32.west)+(0,\ShiftArrow)$);
	\draw[thick, line cap=round, -Stealth] ($(exaAfter32.east)+(-0.4,\ShiftArrow)$) -- node[inner sep=1pt, above]{$\delta_1$} ($(exaFinal.west)+(0,\ShiftArrow)$);

	% Bottom horizontal arrow
	\draw[thick, line cap=round, -Stealth] ($(simInitial.east)+(-0.55,0)$) -- node[inner sep=0pt, above]{2-3($\times2$)} ($(simArch.west)+(0.3,0)$);

	% Bottom diagonal arrows
	\draw[thick, line cap=round, -Stealth] ($(simArch.south)$) -- node[inner sep=1pt, above right]{2-3} ($(simShiftArch.north)$);
	\draw[thick, line cap=round, -Stealth] ($(simShiftArch.north)+(1.2,0)$) -- node[inner sep=0pt, above left]{$\delta_1$} ($(simShift23.south)+(-1.2,0)$);

	\end{tikzpicture}
\caption{Converting the sequence $\gamma,\delta_1$ into 2-3 moves, at the cost of creating an extra arch.}
\label{fig:replacementExample}
\end{figure}

With this idea in mind, let $k\geqslant1$ and suppose we have constructed a suitable sequence $P_0,\ldots,P_{k-1}$.
We need to construct a new spine $P_k$ such that:
\begin{itemize}
\item $P_k$ looks just like $Q_k$ but with an extra arch; and
\item $P_{k-1}$ can be transformed into $P_k$ using 2-3 moves.
\end{itemize}
For this, we know that $Q_{k-1}$ can be turned into $Q_k$ using a single 2-3 move $\delta_k$,
and we also know that $P_{k-1}$ looks like $Q_{k-1}$ with an extra arch.
In the singular graph $S_{Q_{k-1}}$, consider the edge $e$ along which we perform $\delta_k$.
Recall from Section~\ref{subsec:arch} that $e$ corresponds to a chain of edges in $S_{P_{k-1}}$ such that:
\begin{itemize}
\item the chain starts and ends at natural vertices; and
\item every intermediate vertex in the chain is a subdividing vertex.
\end{itemize}
Each subdividing vertex corresponds to a point at which the arch obstructs the 2-3 move $\delta_k$.
Thus, following the idea described above, we first need to shift the arch out of the way using 2-3 moves;
we only require one 2-3 move for each subdividing vertex.
After removing all the subdividing vertices in this way, we can obtain the required spine $P_k$ using a 2-3 move identical to $\delta_k$.

By induction, this completes the construction of a sequence of 2-3 moves that transforms $P$ into a new spine $P_k$.
As explained at the beginning of the proof, all that remains is to turn $P_k$ into $Q$ using a unimodal sequence of 2-3 and 2-0 moves.
Since $P_k$ looks just like $Q$ but with an extra arch, we can achieve this by removing the arch according to Observation~\ref{obs:removeArch}.
\end{proof}

\thmSemiMono*

\begin{proof}
Throughout this proof, call a sequence $\varSigma$ of 2-3, 3-2 and 2-0 moves \textbf{benign} if it consists of two parts:
an arbitrary sequence of 2-3 and 3-2 moves, followed by a (possibly empty) sequence of 2-0 moves.
Thus, a unimodal sequence of 2-3 and 2-0 moves is just a benign sequence with no 3-2 moves.
With this in mind, the idea of this proof is to turn a benign sequence into a unimodal sequence by
repeatedly applying Lemma~\ref{lem:localSemiMono} to reduce the number of 3-2 moves.

To do this, let $P$ be the special spine dual to $\mathcal{T}$, and let $Q$ be the special spine dual to $\mathcal{U}$.
By Amendola's result (Theorem~\ref{thm:Amendola}), we know that $P$ is connected to $Q$ by a sequence of 2-3 and 3-2 moves.
Denote this sequence by $\varSigma_m$, where $m$ is the number of 3-2 moves in this sequence;
there is nothing to prove if $m=0$, so we may assume that $m\geqslant1$ for the rest of this proof.
Since $\varSigma_m$ is a benign sequence (with no 2-0 moves), we can use it as the starting point for our construction.

To inductively reduce the number of 3-2 moves in our benign sequence, fix any particular
$i\in\{1,\ldots,m\}$, and suppose we have constructed a benign sequence $\varSigma_i$ such that:
\begin{itemize}
\item $\varSigma_i$ transforms $P$ into $Q$; and
\item $\varSigma_i$ includes exactly $i$ 3-2 moves.
\end{itemize}
Let $\gamma$ denote the last 3-2 move in $\varSigma_i$.
Because $\varSigma_i$ is benign, $\gamma$ must be followed immediately by 2-3 moves $\delta_1,\ldots,\delta_k$, for some $k\geqslant0$;
these 2-3 moves must then be followed immediately by the sequence of 2-0 moves at the end of $\varSigma_i$.
By Lemma~\ref{lem:localSemiMono}, we can replace the sequence $\gamma,\delta_1,\ldots,\delta_k$ with a unimodal sequence of 2-3 and 2-0 moves;
this produces a benign sequence $\varSigma_{i-1}$ with $(i-1)$ 3-2 moves.

At the end of this inductive construction, we therefore obtain a benign sequence $\varSigma_0$ with no 3-2 moves.
As observed earlier, this means that $\varSigma_0$ is a unimodal sequence of 2-3 and 2-0 moves that transforms $P$ into $Q$;
by duality, this proves the theorem for the triangulations $\mathcal{T}$ and $\mathcal{U}$.
\end{proof}

To conclude this section, we briefly discuss Conjecture~\ref{conj:monotonic}.
The only difference between this conjecture and Theorem~\ref{thm:semiMono} is that we seek a unimodal sequence of
2-3 and 3-2 moves, rather than a unimodal sequence of 2-3 and 2-0 moves.
Given this similarity, it is natural to wonder whether a similar proof strategy could work to prove Conjecture~\ref{conj:monotonic}.

Going back through the proof of Theorem~\ref{thm:semiMono}, and in particular the proof of Lemma~\ref{lem:localSemiMono},
we see that the only place where we use 2-0 moves is when we remove arches according to Observation~\ref{obs:removeArch}.
Thus, one possible way to prove Conjecture~\ref{conj:monotonic} would be to establish the following:

\begin{conjecture}\label{conj:removeArchMono}
An arch can be removed using a unimodal sequence of 2-3 and 3-2 moves.
\end{conjecture}

In contrast to Observation~\ref{obs:removeArch}, which follows immediately from our construction of arches,
Conjecture~\ref{conj:removeArchMono} appears to be rather difficult to prove.

\section{Experimental results}\label{sec:practical}

As discussed in Section~\ref{subsec:introMoves}, when verifying that two triangulations are indeed homeomorphic,
in practice we often need to resort to a blind search for a sequence of Pachner moves.
Thus, one reason to be interested in unimodal sequences is to allow us to perform a more targeted search.
In this section, we compare these ``blind'' and ``targeted'' approaches by testing how efficiently they find sequences
connecting pairs of minimal triangulations of the same $3$-manifold.
Here, a \textbf{minimal triangulation} of a $3$-manifold $\mathcal{M}$ means a triangulation of $\mathcal{M}$ with the smallest possible number of tetrahedra.

Our experimental data set is the census of all minimal triangulations of
closed irreducible $3$-manifolds with $n \leq 11$ tetrahedra~\cite{Burton2011ISSAC}.
We use only manifolds with more than one minimal triangulation;
in the orientable case, we also restrict ourselves to $n \geq 3$ (to avoid special cases) and $n \leq 10$ (to keep the experiments feasible).
The resulting data set contains 2\,628 orientable and 356 non-orientable $3$-manifolds, with 17\,027 minimal triangulations in total.

For each such $3$-manifold, we attempt to connect all of its minimal triangulations with Pachner moves using three algorithms:
\begin{itemize}
\item To find unstructured sequences of Pachner moves, we use a \textbf{blind search} that attempts all possible
sequences of 2-3 and 3-2 moves that never exceed $n+h$ tetrahedra, for increasing values of $h$.
\item To find unimodal sequences of 2-3 and 3-2 moves, we use a more targeted \textbf{unimodal search} that only considers all sequences of $h$ 2-3 moves up from
each minimal triangulation, for increasing values of $h$, until these sequences meet at some common triangulation(s).
\item To find unimodal sequences of 2-3 and 2-0 moves, we use an \textbf{augmented unimodal search} that is similar to the unimodal search,
except that after each individual 2-3 move, we also attempt to connect with other increasing 2-3 paths using all possible sequences of 2-0 moves.
\end{itemize}

Each of these searches can generate enormous numbers of triangulations, leading to high time and memory costs, and so the algorithms must be implemented carefully;
here we use isomorphism signatures to avoid revisiting triangulations, and union-find to detect when sequences of moves intersect.
See~\cite{Burton2011} for details.
All code was written using \emph{Regina}~\cite{Burton2004Regina,Burton2013Regina,Regina}.

We measure the performance of each search algorithm by the total number of triangulations that
it constructs, since this directly determines both the running time and memory usage.
We removed one manifold from our data set because the unimodal search exceeded $50\,000\,000$ triangulations (a limit imposed to avoid exhausting memory).
Our key observations:
\begin{itemize}
\item Apart from the manifold that we removed, we were able to find unimodal sequences connecting all pairs of homeomorphic minimal triangulations in our data set.
This gives some experimental evidence in favour of Conjecture~\ref{conj:monotonic}.
\item In \emph{every} case, the unimodal and augmented unimodal searches both generated the same number of triangulations from
2-3 and 3-2 moves, and required the same additional number of tetrahedra $h$.
This means that in practice, the 2-0 moves are unnecessary, lending further weight to Conjecture~\ref{conj:monotonic}.
\item For $\sim88\%$ of manifolds, none of the triangulations generated in the augmented unimodal search supported 2-0 moves at all.
For the remaining minority of manifolds, the number of additional triangulations generated from 2-0 moves is consistently small, though the portion does rise slowly;
see Figure~\ref{fig:descent}, which plots these numbers on a log scale (each data point represents a single manifold).
This suggests that, though the 2-0 moves are unnecessary, they are also not costly.
\begin{figure}[htbp]
\centering
\includegraphics[scale=0.65]{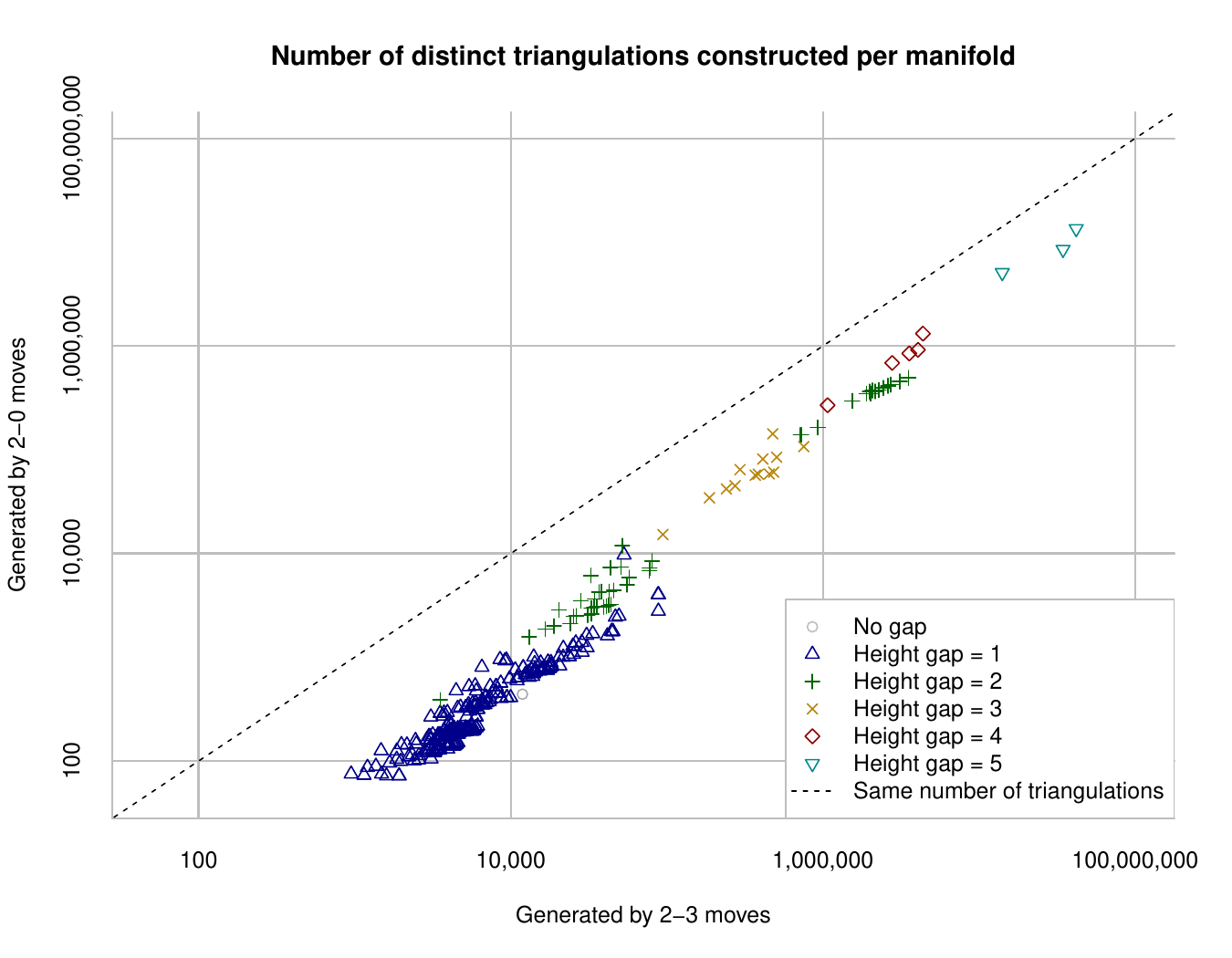}
\caption{Measuring the cost of 2-0 moves in the augmented unimodal search.}
\label{fig:descent}
\end{figure}
\item The single most important factor in the performance is the \textbf{height gap}:
the difference between the number $h$ of additional tetrahedra required in each search type.
Figure~\ref{fig:heights} compares the blind and unimodal searches on a log scale, and categorises the manifolds by this height gap;
we see that each additional tetrahedron costs roughly an additional order of magnitude.
\end{itemize}

\begin{figure}[htbp]
\centering
\includegraphics[scale=0.65]{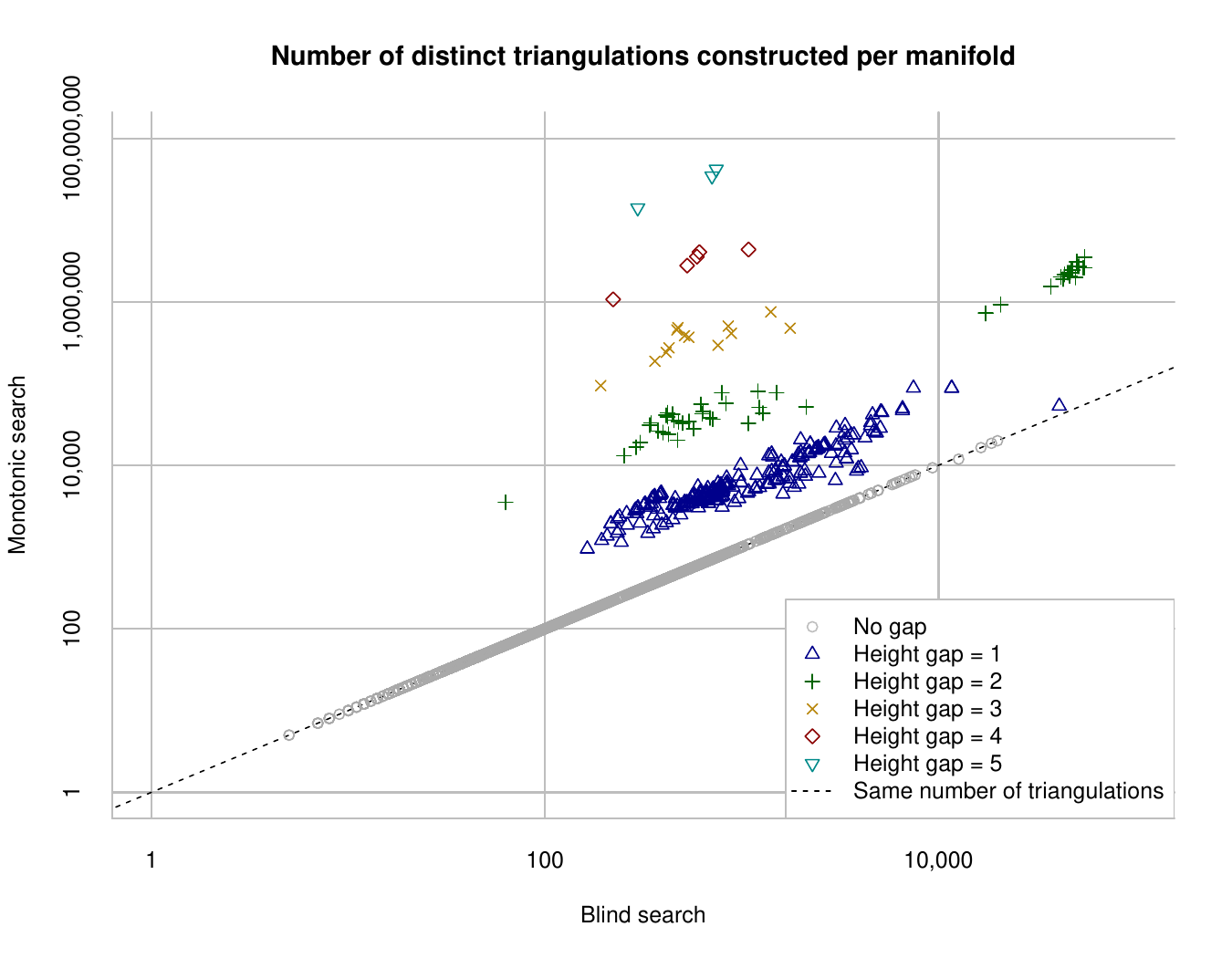}
\caption{Comparing the performance of the blind and unimodal searches.}
\label{fig:heights}
\end{figure}

From these experiments, the clear message is: when attempting to connect two triangulations in practice to prove that they are homeomorphic,
we should use a blind search, because the cost of searching through a larger set of triangulations with
$n+h$ tetrahedra is found to be significantly cheaper than the cost of extending the search to $n+h+1$ tetrahedra.

This also highlights the importance of experimental testing: the blind search could theoretically run through $e^{O(n \log n)}$ triangulations at every new height,
whereas the unimodal search only searches through $O(n^h)$ triangulations at height $h$.
Although the former cost is more expensive in theory, our experiments show that it is cheaper in practice.

This suggests that the main benefit of unimodal sequences is theoretical: when working in a proof with two triangulations of the same $3$-manifold,
instead of just assuming there is an arbitrary path of 2-3 and 3-2 moves with no particular structure,
we may strengthen this assumption to use a unimodal path of 2-3 and 2-0 moves instead.

Having said this, we can make two final observations:
\begin{itemize}
\item In most of our test cases, the height gap is small: of our 2984 manifolds, only 81 had a height gap over $2$.
Thus, the performance gap we see, though significant when it occurs, appears uncommon.
\item In \emph{all} of our test cases, the heights themselves are surprisingly small:
$\leq 3$ for every blind search, and $\leq 7$ for every unimodal search (except for the one that was terminated early).
\end{itemize}

It is also important to note that our experiments may be subject to the tyranny of small numbers:
whilst we cover hundreds of millions of intermediate triangulations, these triangulations are all relatively small.
Due to the exponential time and memory requirements, obtaining exact data such as ours quickly becomes impossible as the number of tetrahedra grows.

%%
%% Bibliography
%%
\bibliography{MonotonicRefs}

\end{document}